\documentclass[11pt]{amsart}
\usepackage[utf8]{inputenc}
\usepackage{amsmath}
\usepackage{amssymb}
\usepackage{amsthm}
\usepackage[all]{xy}
\usepackage{graphicx}
\usepackage[margin=1in]{geometry}
\usepackage{color}
\usepackage{hyperref}
\usepackage{verbatim}
\usepackage{pb-diagram,pb-xy}

\numberwithin{equation}{section}

\newcommand{\R}{\ensuremath{\mathbb{R}}}
\newcommand{\N}{\ensuremath{\mathbb{N}}}

\newcommand{\Z}{\ensuremath{\mathbb{Z}}}

\newtheorem{definition}{Definition}[section]

\newtheorem{theorem}{Theorem}[section]
\newtheorem{corollary}{Corollary}[theorem]
\newtheorem{proposition}[theorem]{Proposition}
\newtheorem{lemma}[theorem]{Lemma}

\theoremstyle{definition} 
\newtheorem{remark}[theorem]{Remark} 

\newtheorem{Construction}{Construction}[section]

\newcommand{\mb}[1]{\mathbf{#1}}

\DeclareMathOperator*{\colim}{colim}

\DeclareMathOperator*{\Ob}{Ob}

\DeclareMathOperator*{\Int}{Int}

\DeclareMathOperator*{\Diff}{Diff}

\DeclareMathOperator*{\Emb}{Emb}
\DeclareMathOperator*{\BDiff}{BDiff}

\DeclareMathOperator*{\Ab}{Ab}

\DeclareMathOperator*{\lk}{\text{lk}}
\DeclareMathOperator*{\fr}{\text{fr}}
\DeclareMathOperator*{\Id}{\text{Id}}
\DeclareMathOperator*{\Bun}{Bun}
\DeclareMathOperator*{\Fr}{Fr}

\begin{document}

\title{Homological Stability for diffeomorphism groups of high dimensional handlebodies}

\author{Nathan Perlmutter}

  \email{nperlmut@stanford.edu}
  
  \address{Stanford University Department of Mathematics, Building 380, Stanford, California,  94305, USA} 

\begin{abstract}
In this paper we prove a homological stability theorem for the diffeomorphism groups of high dimensional manifolds with boundary, with respect to forming the boundary connected sum with the product $D^{p+1}\times S^{q}$ for $|q - p| < \min\{p, q\} - 2$. 
In a recent joint paper with B. Botvinnik, we prove that there is an isomorphism  
$$\colim_{g\to \infty}H_{*}(\BDiff((D^{n+1}\times S^{n})^{\natural g}, \; D^{2n}); \Z) \; \cong \; H_{*}(Q_{0}BO(2n+1)\langle n\rangle_{+}; \Z)$$
in the case that $n \geq 4$. 
By combining this ``stable homology'' calculation with the homological stability theorem of this current paper, we obtain the isomorphism $H_{k}(\BDiff((D^{n+1}\times S^{n})^{\natural g}, \; D^{2n}); \Z) \cong  H_{k}(Q_{0}BO(2n+1)\langle n\rangle_{+}; \Z)$ in the case that $k \leq \tfrac{1}{2}(g - 4)$.
\end{abstract}

\maketitle

\section{Introduction} \label{Introduction} 
\subsection{Main result}
Let $M$ be a smooth, compact, $m$-dimensional manifold with non-empty boundary. 
Fix an $(m-1)$-dimensional disk $D^{m-1} \hookrightarrow \partial M$. 
We denote by $\Diff(M, D^{m-1})$ the group of self-diffeomorphisms of $M$ that restrict to the identity on a neighborhood of $D^{m-1}$, topologized in the $C^{\infty}$-topology.
Let $\BDiff(M, D^{m-1})$ denote the \textit{classifying space} of the topological group $\Diff(M, D^{m-1})$. 
Choose an $m$-dimensional manifold $V$ with non-empty boundary and let $M\natural V$ denote the \textit{boundary connected sum} of $M$ and $V$. 
There is a homomorphism $\Diff(M, D^{m-1}) \longrightarrow \Diff(M\natural V, D^{m-1})$ defined by extending a diffeomorphism identically over the boundary connect-summand $V$. 
This homomorphism induces a map between the classifying spaces, and iterating this map yields the direct system,
\begin{equation} \label{equation: direct system}
\xymatrix{
\BDiff(M, D^{m-1}) \ar[r] & \BDiff(M\natural V, D^{m-1}) \ar[r] & \cdots \ar[r] & \BDiff(M\natural V^{\natural g}, D^{m-1}) \ar[r] & \cdots 
}
\end{equation}
In this paper we study the homological properties of this direct system in the case when $V$ is a high-dimensional handlebody $D^{p+1}\times S^{q}$. 
Our main result can be viewed as an analogue of the homological stability theorems of Harer \cite{Ha 85}, or the theorem of Galatius and Randal-Williams \cite{GRW 14}, for high-dimensional manifolds with boundary. 

Below we state our main theorem.
For $p, q, g \in \Z_{\geq 0}$, let $V^{g}_{p, q}$ denote the $g$-fold boundary connected sum, $(D^{p+1}\times S^{q})^{\natural g}$. 
Let $d(\pi_{q}(S^{p}))$ denote the \textit{generating set length} of the homotopy group $\pi_{q}(S^{p})$, which is the quantity,
$
d(\pi_{q}(S^{p})) \; = \; \min\{ k \in \N \; | \; \text{there exists an epimorphism} \; \Z^{\oplus k} \rightarrow \pi_{q}(S^{p}) \}.
$
We let $\kappa(\partial M)$ and $\kappa(M, \partial M)$ denote the degrees of connectivity of $\partial M$ and $(M, \partial M)$ respectively. 
The main result of this paper is stated below.
\begin{theorem} \label{theorem: main theorem 1}
Let $M$ be a compact manifold of dimension $m$ with non-empty boundary. 
Let $p$ and $q$ be positive integers with $p + q + 1 = m$, such that inequalities,
\begin{equation} \label{equation: required theorem inequalities}
|q - p| < \min\{p, q\} - 2 \quad \text{and} \quad |q - p|  < \min\{\kappa(\partial M), \kappa(M, \partial M)\} - 1
\end{equation}
are satisfied. 
Then the homomorphism induced by the maps in (\ref{equation: direct system}),
$$
H_{k}(\BDiff(M\natural V^{g}_{p, q}, D^{m-1}); \; \Z) \; \longrightarrow \; H_{k}(\BDiff(M\natural V^{g+1}_{p, q}, D^{m-1}); \; \Z),
$$
is an isomorphism when $k \leq \frac{1}{2}(g - d(\pi_{q}(S^{p})) - 3)$ and an epimorphism when $k \leq \frac{1}{2}(g - d(\pi_{q}(S^{p})) - 1)$.
\end{theorem}

The above theorem implies that the homology of the direct system (\ref{equation: direct system}) stabilizes in the case $V = V_{p,q}^{1}$. 
For certain choices of $p$ and $q$ we can also identify the homology of the colimit. 
In joint work with B. Botvinnik \cite{BP 15} we construct a map, 
$$
\colim_{g\to\infty}\BDiff((D^{n+1}\times S^{n})^{\natural g}, D^{2n}) \longrightarrow Q_{0}BO(2n+1)\langle n \rangle_{+},
$$
which we prove induces an isomorphism in $H_{*}(\--, \Z)$ in the case that $n \geq 4$. 
Combining this homological equivalence with Theorem \ref{theorem: main theorem 1} we obtain the following corollary which lets us compute the homology of the classifying space $\BDiff((D^{n+1}\times S^{n})^{\natural g}, D^{2n})$ in low degrees relative to $g$.
\begin{corollary}
Let $2n+1 \geq 9$. 
Then there is an isomorphism, 
$$
H_{k}(\BDiff((D^{n+1}\times S^{n})^{\natural g}, D^{2n}); \; \Z) \; \cong \; H_{k}(Q_{0}BO(2n+1)\langle n \rangle_{+}; \; \Z),
$$
when $k \leq \tfrac{1}{2}(g - 4)$. 
\end{corollary}

In addition to the classifying spaces of diffeomorphism groups of manifolds, we prove an analogous homological stability theorem for the moduli spaces of manifolds equipped with tangential structures, see Theorem \ref{theorem: main homological stability theorem theta structure}.
The precise statement of this theorem requires a number of preliminary definitions and so we hold off on stating this result until Section \ref{section: tangential structures} where the theorem is proven.

\subsection{Ideas behind the proof} 
The proof or our main theorem follows the strategy developed by Galatius and Randal-Williams in \cite{GRW 14} used to prove homological stability for the diffeomorphism groups, $\Diff((S^{n}\times S^{n})^{\#g}, D^{2n})$, $g \in \N$. 
The technique is also largely inspired by analogous the homological stability theorem of Harer \cite{Ha 85} for the mapping class groups of surfaces. 
Given a compact manifold triad $(M, \partial_{0}M, \partial_{1}M)$ and integers $p+q+1 = \dim(M)$, we construct a simplicial complex $K^{\partial}(M)_{p,q}$, that admits an action of the diffeomorphism group $\Diff(M, \partial_{0}M)$. 
Roughly, an $l$-simplex of $K^{\partial}_{\bullet}(M)_{p,q}$ is given by a set of embeddings,
$\phi_{0}, \dots, \phi_{l}: (V^{1}_{p,q}, \partial V^{1}_{p,q}) \longrightarrow (M, \partial_{1}M),$ 
for which $\phi_{i}(V^{1}_{p,q})\cap\phi_{j}(V^{1}_{p,q}) = \emptyset$ for all $i \neq j$.
Nearly all of the technical work of this paper is devoted to showing that the connectivity of the geometric realization $|K^{\partial}(M)_{p,q}|$ increases linearly (with a factor of $1/2$) in the number of boundary-connect-summands of $V_{p,q}^{1}$ contained in $M$ (see Theorem \ref{theorem: high-connectivity}). 
Once high-connectivity is established, the homological stability theorem follows by the same spectral sequence argument used in \cite{GRW 14}.
Most of the simplicial and homotopy theoretic techniques that we use come from \cite{GRW 14}, but the geometric-topological the aspects of the paper require new constructions and arguments. 
We describe some of these new features below.

The simplicial complex $K^{\partial}_{\bullet}(M)_{p,q}$ can be viewed as a ``relative version'' of the semi-simplicial space constructed in our previous work \cite{P 15}, used to prove homological stability for the diffeomorphism groups $\Diff((S^{q}\times S^{p})^{\# g})$, $g \in \N$.
Our proof that the space $|K^{\partial}(M)_{p,q}|$ is highly connected follows a similar strategy to what was employed in \cite{P 15}.
The idea is to map $K^{\partial}(M)_{p,q}$ to another simplicial complex, $L(\mathcal{W}^{\partial}_{p,q}(M, \partial_{1}M))$, that is constructed entirely out of algebraic data associated to the manifold $M$, which we call the \textit{Wall form} associated to $M$ (see Section \ref{section: Algebra}). 
Wall forms were initially defined in \cite{P 15} to study closed manifolds and their diffeomorphisms.
In this paper we have to use a relative version of the Wall form used to study manifold pairs $(M, \partial M)$ and diffeomorphisms $M \longrightarrow M$ that are non-trivial on the boundary.
High-connectivity of the complex $L(\mathcal{W}^{\partial}_{p,q}(M, \partial_{1}M))$ follows from our work in \cite{P 15} as well, and so it remains to prove that the map $\bar{K}^{\partial}_{\bullet}(M)_{p,q} \longrightarrow L(\mathcal{W}^{\partial}_{p,q}(M, \partial_{1}M))$ is highly connected. 

One of the main ingredients used in our proof that the map $\bar{K}^{\partial}_{\bullet}(M)_{p,q} \longrightarrow L(\mathcal{W}^{\partial}_{p,q}(M, \partial_{1}M))$ is highly-connected is a new disjunction result for embeddings of high-dimensional manifolds with boundary, this is Theorem \ref{theorem: boundary disjunction} stated in the appendix. 
If $(P, \partial P), (Q, \partial Q) \; \subset \; (M, \partial M)$ are submanifold pairs, Theorem \ref{theorem: boundary disjunction} gives conditions for when one can find an ambient isotopy $\Psi_{t}: (M, \partial M) \longrightarrow (M, \partial M)$ with $\Psi_{0} = \Id_{M}$, such that $\Psi_{1}(P)\cap Q = \emptyset$.
Theorem \ref{theorem: boundary disjunction} can be viewed as a generalization of the disjunction results of Wells \cite{We 67}, and Hatcher and Quinn \cite{HQ 74}, and thus could be of independent interest. 

\subsection{Plan of the paper}
In Section \ref{section: Preliminary Constructions} we recast the maps of the direct system (\ref{equation: direct system}) as maps arising from concatenation with a relative cobordism between two compact manifold pairs.
This will enable us to restate Theorem \ref{theorem: main theorem 1} in a slightly more general form that will be easier for us to prove. 
In Section \ref{section: simplicial complexes} we construct the main simplicial complex, $K^{\partial}(M)_{p,q}$.
In Section \ref{Section: The Algebraic Invariants} we define certain algebraic invariants associated to a manifold with boundary. 
Section \ref{section: Algebra} is devoted to a recollection of some results from our previous work in \cite{P 15a} regarding Wall forms.
In Section \ref{section: high connectivity of K} we use the results developed throughout the rest of the paper to prove that $|K^{\partial}(M)_{p,q}|$ is highly connected.
In Section \ref{section: homological stability}, we show how to obtain Theorem \ref{theorem: main homological stability theorem} (and thus Theorem \ref{theorem: main theorem 1}) using the high connectivity of $|K^{\partial}(M)_{p,q}|$.
In Section \ref{section: tangential structures} we show how to obtain an analogue of Theorem \ref{theorem: main homological stability theorem} for the moduli spaces of handlebodies equipped with tangential structures.  
In Appendix \ref{appendix: embeddings and disjunction} we prove a disjunction theorem for embeddings of high-dimensional manifolds with boundary.

\subsection{Acknowledgments}
The author thanks Boris Botvinnik for helpful conversations.
The author was supported by NSF Postdoctoral Fellowship, DMS-1502657.

\section{Relative Cobordism and Stabilization} \label{section: Preliminary Constructions}
To prove Theorem \ref{theorem: main theorem 1}, we will need to recast the direct system (\ref{equation: direct system}) as arising from concatenation with a relative cobordism between two compact manifold pairs.
Doing this will make our constructions consistent with the cobordism categories considered in \cite{G 12} and \cite{BP 15}.
Below we introduce some definitions and terminology that we will use throughout the paper. 
Recall that a \textit{triad} of topological spaces is a triple $(X; A, B)$ where $X$ is a topological space and $A, B \subset X$ are subspaces. 
\begin{definition} \label{defn: manifold triad}
A \textit{manifold triad} of dimension $n$ is a triad $(W; \partial_{0}W, \partial_{1}W)$ where $W$ is an $n$-dimensional smooth manifold, $\partial_{0}W, \partial_{1}W \subset \partial W$ are submanifolds of dimension $n-1$ such that $\partial W = \partial_{0}W\cup\partial_{1}W$ and $\partial(\partial_{0}W) = \partial_{0}W\cap\partial_{1}W = \partial(\partial_{1}W)$.
We will denote by $\partial_{0,1}W$ the intersection $\partial_{0}W\cap\partial_{1}W$. 
We will refer to $\partial_{0}W$ and $\partial_{1}W$ as the \textit{faces}. 

Two compact $d$-dimensional manifold pairs 
$(M, \partial M)$ and $(N, \partial N)$ are said to be \textit{cobordant} if there exists a 
$(d+1)$-dimensional compact manifold triad
$(W; \partial_{0}W, \partial_{1}W)$ such that $(\partial_{0}W, \partial_{0,1}W) \; =  \;  (M\sqcup N, \; \partial M\sqcup \partial N)$. 
The pair $(W, \partial_{1}W)$ is then said to be a \textit{relative cobordism} between the pairs $(M, \partial M)$ and $(N, \partial N)$. 
\end{definition}

Let $(M; \partial_{0}M, \partial_{1}M)$ be an $m$-dimensional, compact manifold triad.
Let $\Diff(M)$ denote the topological group of self-diffeomorphisms $f: M \longrightarrow M$ with $f(\partial_{i}M) = \partial_{i}M$ for $i = 0, 1$. 
We are mainly interested in the subgroup $\Diff(M, \partial_{0}M) \subset \Diff(M)$ consisting of those self-diffeomorphisms that restrict to the identity on a neighborhood of $\partial_{0}M$. 
We will need the following construction.
\begin{Construction} \label{construction: the stabilization map}
Let $(P, \partial P)$ be an $(m-1)$-dimensional manifold pair and
let $(K; \partial_{0}K, \partial_{1}K)$ be a compact manifold triad such that $\partial_{0}K = \partial_{0}M\sqcup P$, i.e.\ the pair $(K, \partial_{1}K)$ is a relative cobordism between $(\partial_{0}M, \partial_{0, 1}M)$ and $(P, \partial P)$.
Let $M\cup_{\partial_{0}}K$ be the manifold obtained by attaching $K$ to $M$ along the face $\partial_{0}M \subset \partial_{0}K$.
Similarly, let $\partial_{1}(M\cup_{\partial_{0}}K) := \partial_{1}M\cup_{\partial_{0,1}}\partial_{1}K$ be the manifold obtained by attaching
$\partial_{1}K$ to $\partial_{1}M$ along their common boundary $\partial_{0,1}M$.
 By setting 
$\partial_{0}(M\cup_{\partial_{0}}K) := P,$ we obtain a new manifold triad, 
$$\left(M\cup_{\partial_{0}}K; \; \partial_{0}(M\cup_{\partial_{0}}K), \;  \partial_{1}(M\cup_{\partial_{0}}K)\right).$$
The cobordism $(K, \partial_{1}K)$ induces a homomorphism, 
$\Diff(M, \partial_{0}M) \longrightarrow \Diff(M\cup_{\partial_{0}}K, P)$, $f \mapsto f\cup\textstyle{\Id_{K}},$
defined by extending diffeomorphisms identically over $K$. 
This homomorphism in turn induces a map on the level of classifying spaces, 
\begin{equation} \label{equation: stabilization map cobordism}
\BDiff(M, \partial_{0}M) \longrightarrow \BDiff(M\cup_{\partial_{0}M}K, P).
\end{equation}
This map should be compared to the one from \cite[Equation 1.1]{GRW 14}.
Indeed, they are the same map in the case that $\partial_{1}M = \partial_{1}K = \emptyset$.
\end{Construction}

Suppose now that $\partial_{0}M$ and $\partial_{1}M$ are non-empty. 
Let $p$ and $q$ be integers such that $p + q + 1 = m$. 
Let $K_{p, q}$ then denote the manifold obtained by forming the boundary connected sum of $D^{p+1}\times S^{q}$, with $\partial_{0}M\times[0, 1]$.
In constructing $K_{p,q}$ we form the boundary connected sum along a neighborhood contained in the complement of the subspace 
$\partial_{0}M\times\{0,1\} \; \subset \; \partial_{0}M\times[0, 1];$
in this way the boundary of $K_{p,q}$ will contain $\partial_{0}M\times\{0,1\}$.
We then set: 
$$
\partial_{0}K_{p,q} := \partial_{0}M\times\{0,1\}, \quad  \quad \partial_{1}K_{p,q} := \partial K_{p,q}\setminus \Int(\partial_{0}K_{p,q}) \; = \; (\partial_{0, 1}M\times[0,1])\#(S^{p}\times S^{q}).
$$
The triple $(K_{p,q}; \; \partial_{0}K_{p,q}, \; \partial_{1}K_{p,q})$ is a manifold triad and $(K_{p,q}, \partial_{1}K_{p,q})$ is a relative cobordism between the manifold pairs $(\partial_{0}M\times\{0\},\; \partial_{0,1}M\times\{0\})$ and $(\partial_{0}M\times\{1\}, \; \partial_{0,1}M\times\{1\})$.
We apply Construction \ref{construction: the stabilization map} to this relative cobordism $(K_{p,q}, \partial_{1}K_{p,q})$. 
We form the manifold $M\cup_{\partial_{0}}K_{p,q}$ and notice that $\partial_{0}(M\cup_{\partial_{0}}K_{p,q}) = \partial_{0}M$.
We define 
\begin{equation} \label{equation: stabilization map}
\xymatrix{
s_{p,q}: \BDiff(M, \; \partial_{0}M) \; \longrightarrow \; \BDiff(M\cup_{\partial_{0}}K_{p,q}, \; \partial_{0}M)
}
\end{equation}
to be the map on classifying spaces induced by $(K_{p,q}, \partial_{1}K_{p,q})$ from (\ref{equation: stabilization map cobordism}).
We will refer to this map as the \textit{$(p, q)$-th stabilization map}.

\begin{remark}
Note that the manifold $M\cup_{\partial_{0}}K_{p,q}$ is diffeomorphic to the boundary connected sum $M\natural V^{1}_{p,q} = M\natural(D^{p+1}\times S^{q})$. 
By identifying $M\cup_{\partial_{0}}K_{p,q}$ with $M\natural V^{1}_{p,q}$, (\ref{equation: stabilization map}) yields the map,
$$
\BDiff(M, \; \partial_{0}M) \longrightarrow \BDiff(M\natural V^{1}_{p,q}, \; \partial_{0}M).
$$
In the case that $\partial_{0}M = D^{p+q}$ we obtain the maps used in the direct system (\ref{equation: direct system}); we take this to be the definition of those maps.
\end{remark}

We will now restate Theorem \ref{theorem: main theorem 1} in terms of the maps defined above. 
We need one more preliminary definition. 
Recall from the previous section the manifold $V^{g}_{p,q} = (D^{p+1}\times S^{q})^{\natural g}$.
Choose an embedded $(p+q)$-dimensional closed disk $D \subset \partial V^{g}_{p,q}$ and set 
$\partial_{0}V^{g}_{p,q} := D$ and $\partial_{1}V^{g}_{p,q} := \partial V^{g}_{p,q}\setminus\Int(D)$
so as to obtain a manifold triad.
With $\dim(M) = m$ and $p+q + 1 = m$ as above, we let $r_{p,q}(M)$ be the integer defined by,
\begin{equation} \label{equation: p-q rank of mfd}
r_{p,q}(M) \; = \; \max\{ g \in \N \; | \; \text{there exists an embedding $(V^{g}_{p,q}, \partial_{1}V^{g}_{p,q}) \longrightarrow (M, \partial_{1}M)$}  \}.
\end{equation}
We refer to this quantity as the $(p, q)$-rank of $M$. 
This quantity $r_{p,q}(M)$ is equivalent to the maximal number of boundary connect summands of $D^{p+1}\times S^{q}$ that split off of $M$. 
We emphasize that the embeddings $(V^{g}_{p,q}, \partial_{1}V^{g}_{p,q}) \longrightarrow (M, \partial_{1}M)$ used in
(\ref{equation: p-q rank of mfd}) need not send the face $\partial_{0}V^{g}_{p,q}$ into $\partial_{0}M$.
\begin{remark}
The value $r_{p,q}(M)$ depends on the structure of the triad $(M; \partial_{0}M, \partial_{1}M)$ and not just the manifold $M$ itself. 
In particular, switching the roles of $\partial_{0}M$ and $\partial_{1}M$ in (\ref{equation: p-q rank of mfd}) will change the value of the rank $r_{p,q}(M)$. 
\end{remark}
As in the statement of Theorem \ref{theorem: main theorem 1} we will need to assume that the following inequalities are satisfied:
\begin{equation} \label{equation: fundamental inequalities}
|q - p| < \min\{p, q\} - 2, \quad |q - p|  < \min\{\kappa(\partial_{1}M), \kappa(M, \partial_{1}M)\} - 1,
\end{equation}
where recall that $\kappa(\partial_{1}M)$ and $\kappa(M, \partial_{1}M)$ denote the degrees of connectivity of $\partial_{1}M$ and $(M, \partial_{1}M)$ respectively.
The main theorem that we will prove in this paper is stated below; 
it is a generalization of Theorem \ref{theorem: main theorem 1}. 
Recall the generating set length $d(\pi_{q}(S^{p}))$. 
\begin{theorem} \label{theorem: main homological stability theorem}
Let $(M; \partial_{0}M, \partial_{1}M)$ be an $m$-dimensional, compact, manifold triad with $\partial_{0}M$ and $\partial_{1}M$ nonempty. 
Let $p$ and $q$ be positive integers with $p + q +1 = m$ and suppose that the inequalities of (\ref{equation: fundamental inequalities}) are satisfied. 
Let $r_{p,q}(M) \geq g$. 
Then the homomorphism
$$\xymatrix{
(s_{p, q})_{*}: H_{k}(\BDiff(M, \partial_{0}M); \; \Z) \longrightarrow H_{k}(\BDiff(M\cup_{\partial_{0}}K_{p,q}, \; \partial_{0}M); \; \Z)
}$$
is an isomorphism when $k \leq \tfrac{1}{2}(g - 3 - d(\pi_{q}(S^{p})))$ and an epimorphism when $k \leq \tfrac{1}{2}(g - 1 - d(\pi_{q}(S^{p})))$.
\end{theorem}

\section{The Complex of Embedded Handles} \label{section: simplicial complexes} 
In this section we construct a semi-simplicial space that is analogous to \cite[Definition 5.1]{GRW 14}.
The definition requires a preliminary geometric construction. 
For integers $p, q$, and $g$, let $(V_{p,q}^{g}; \; \partial_{0}V_{p,q}^{g}, \partial_{1}V^{g}_{p,q})$ be the $(p+q+1)$-dimensional manifold triad defined in the previous section. 
For each $g \in \N$ we let $W_{p, q}^{g}$ denote the the face, 
$\partial_{1}V^{g}_{p, q} \cong (S^{p}\times S^{q})^{\# g}\setminus\Int(D^{p+q}).$
We will denote $V_{p, q} := V^{1}_{p,q}$ and $W_{p, q} := W^{1}_{p, q}$. 
\begin{definition}
For an integer $m$ let $D^{m}_{+}$ denote the $m$-dimensional half-disk, i.e.\ the subspace given by the set
$\{ \bar{x} \in \R^{m} \; | \; |\bar{x}| \leq 1, \quad x_{1} \geq 0 \; \}.$
The boundary of $D^{m}_{+}$ has the decomposition $\partial D^{m}_{+} = \partial_{0}D^{m}_{+}\cup\partial_{1}D^{m}_{+}$
where $\partial_{0}D^{m}_{+}$ and $\partial_{1}D^{m}_{+}$ are given by,
$\partial_{0}D^{m}_{+} = \{ \bar{x} \in D^{m}_{+} \; | \; x_{1} = 0 \; \}$
and
$\partial_{1}D^{m}_{+} = \{ \bar{x} \in D^{m}_{+} \; | \; |\bar{x}| = 1 \; \}$.
In this way $(D^{m}_{+}; \; \partial_{0}D^{m}_{+}, \; \partial_{1}D^{m}_{+})$ forms a manifold triad. 
\end{definition}
We will construct a slight modification of the manifold $V_{p,q}$. 
Choose an embedding, 
$$
\alpha: (D^{p+q}_{+}\times\{1\}, \; \partial_{0}D^{p+q}_{+}\times\{1\}) \; \longrightarrow \; (\partial_{0}V_{p, q}, \; \partial_{0, 1}V_{p,q}).
$$
Let $\widehat{V}_{p, q}$ denote the manifold obtained by attaching $D^{p+q}_{+}\times[0, 1]$ to $V_{p, q}$ along the embedding $\alpha$, i.e.\ 
$$
\widehat{V}_{p, q} \; = \; V_{p, q}\cup_{\alpha}(D^{p+q}_{+}\times[0, 1]).
$$
We then denote by $\widehat{W}_{p, q}$ the manifold obtained by attaching $\partial_{0}D^{p+q}\times[0,1]$ to $\partial_{1}V_{p, q}$ along the restriction of $\alpha$ to $\partial_{0}D^{p+q}\times\{1\}$. 

\begin{Construction}  \label{construction: main probing manifold}
We construct a subspace of $\widehat{V}_{p, q}$ as follows. 
Choose a basepoint 
$(a_{0}, b_{0}) \in \partial D^{p+1}\times S^{q}$
such that the pair $(a_{0}, -b_{0})$ is contained in the face $\partial_{1}V_{p, q}$. 
Choose an embedding
$
\gamma: [0, 1] \; \longrightarrow \; \widehat{W}_{p, q}
$
that satisfies the following conditions:
\begin{enumerate} \itemsep.2cm
\item[(i)] $\gamma(1) = (a_{0}, -b_{0})$ and $\gamma(0) = (0,0) \in \partial_{0}D^{p+q}_{+}\times\{0\}$;
\item[(ii)] 
there exists $\varepsilon \in (0, 1)$ such that, $\gamma(t) \; = \; (0, t) \in \partial_{0}D^{p+q}_{+}\times[0,1],$
whenever $t \in (0, \varepsilon)$;
\item[(iii)] 
$\gamma((0, 1))\cap\left[(D^{p+1}\times \{a_{0}\})\cup (\{b_{0}\}\times S^{q})\right] = \emptyset$.
\end{enumerate}
We define $(B_{p,q}, C_{p, q}) \subset (\widehat{V}_{p,q}, \widehat{W}_{p, q})$ to be the pair of subspaces given by,
$$
\left(\gamma([0,1])\cup (D^{p+1}\times \{b_{0}\})\cup (\{a_{0}\}\times S^{q}), \; \; \gamma([0,1])\cup (S^{p}\times \{b_{0}\})\cup (\{a_{0}\}\times S^{q})\right).
$$
The inclusion
$(B_{p,q}, C_{p, q}) \hookrightarrow (\widehat{V}_{p,q}, \widehat{W}_{p, q})$
is clearly a homotopy equivalence of pairs. 
We will refer to the pair $(B_{p,q}, C_{p, q})$ as the \textit{core} of $(\widehat{V}_{p,q}, \widehat{W}_{p, q})$. 
\end{Construction}

We now use the above construction to construct a simplicial complex. 
Let $(M; \partial_{0}M, \partial_{1}M)$ be a compact manifold triad of dimension $m$. 
Let $\R^{m-1}_{+}$ denote $[0,\infty)\times\R^{m-2}$ and let $\partial\R^{m-1}_{+}$ denote $\{0\}\times\R^{m-2}$. 
Choose an embedding,
$a: [0,1)\times\R^{m-1}_{+} \longrightarrow M$
with,
$a^{-1}(\partial_{0}M) =  \{0\}\times\R^{m-1}_{+}$ and $a^{-1}(\partial_{1}M) =  [0, 1)\times\partial\R^{m-1}_{+}.$
For each pair of positive integers $p$ and $q$ with $p+q+1 = m$,
we define a simplicial complex
  $K^{\partial}(M, a)_{p,q}$. 
\begin{definition} \label{defn: the embedding complex} 
The simplicial complex $K^{\partial}(M, a)_{p,q}$ is defined as follows:
\begin{enumerate}  \itemsep.2cm
\item[(i)] 
A vertex in $K^{\partial}(M, a)_{p,q}$ is defined to be a pair $(t, \phi)$, where $t \in \R$ and 
$\phi: \widehat{V}_{p,q} \longrightarrow M$
is an embedding that satisfies:
\begin{enumerate} \itemsep.1cm
\item[(a)] $\phi^{-1}(\partial_{0}M) = D^{p+q}_{+}\times\{0\}$ \;  and \;  $\phi^{-1}(\partial_{1}M) = \widehat{W}_{p, q}$; 
\item[(b)]
 there exists $\varepsilon > 0$ such that for
$(s, z) \in [0, \varepsilon)\times D^{m-1}_{+} \; \subset \; \widehat{V}_{p,q}$, \; 
the equality $\phi(s, z) = a(s, z + te_{2})$ is satisfied,
where $e_{2} \in \R^{m-1}_{+}$ denotes the second basis vector. 
\end{enumerate}
\item[(ii)] A set of vertices $\{(\phi_{0}, t_{0}), \dots, (\phi_{l}, t_{l})\}$ forms an $l$-simplex 
if 
$t_{i} \neq t_{j}$ and 
$\phi_{i}(B_{p,q}) \cap \phi_{j}(B_{p,q}) = \emptyset$ whenever $i \neq j$, where recall that $B_{p, q} \subset \widehat{V}_{p, q}$ is the core from Construction \ref{construction: main probing manifold}.  
\end{enumerate}
\end{definition}

\begin{remark} \label{remark: drop a from notation}
The embedding $a$ was necessary in order to define the complex $K^{\partial}(M, a)_{p,q}$, however, it will not play a serious role in any of our proofs latter in the paper. 
For this reason we will drop the embedding $a$ from the notation and will denote 
$K^{\partial}(M)_{p,q} := K^{\partial}(M, a)_{p,q}.$
\end{remark} 
The majority of the technical work of this paper is devoted to proving Theorem \ref{theorem: high-connectivity} stated below. 
The statement of this theorem will require the use of a definition from \cite{GRW 14} which we recall below.
\begin{definition} \label{defn: cohen macaulay}
A simplicial complex $X$ is said to be
\textit{weakly Cohen-Macaulay} of dimension $n$ if it is
$(n-1)$-connected and the link of any $p$-simplex is
$(n-p-2)$-connected. In this case we write $wCM(X) \geq
n$. The complex $X$ is said to be \textit{locally weakly
Cohen-Macaulay} of dimension $n$ if the link of any simplex is
$(n - p - 2)$-connected (but no global connectivity is required on
$X$ itself). In this case we shall write $lCM(X) \geq n$.
\end{definition}

\begin{theorem} \label{theorem: high-connectivity}
Let $(M; \partial_{0}M, \partial_{1}M)$ be an $m$-dimensional manifold triad with $\partial_{0}M \neq \emptyset$. 
Let $p, q \in \N$ be such that $p + q + 1 = m$ and suppose that 
the inequalities, 
\begin{equation} \label{equation: main inequalities restated}
|q - p| < \min\{p, q\} - 2, \quad |q - p|  < \min\{\kappa(\partial_{1}M), \kappa(M, \partial_{1}M)\} - 1,
\end{equation}
are satisfied. 
Let $d = d(\pi_{q}(S^{p}))$ be the generating set rank and suppose that $r_{p,q}(M) \geq g$.
Then the geometric realization $|K^{\partial}(M)_{p,q}|$ is $\tfrac{1}{2}(g-4-d)$-connected and 
$lCM(K^{\partial}(M)_{p,q}) \geq \tfrac{1}{2}(g - 1- d)$.
\end{theorem}

The proof of the above theorem takes place over the course of the next three sections of the paper. 
In addition to the simplicial complex defined above, 
we will need to work with two related semi-simplicial spaces.
Let $(M; \partial_{0}M, \partial_{1}M)$, $p$, and $q$ be as above and let 
$a: [0,1)\times\R^{m-1}_{+} \longrightarrow M$
be the same embedding used in Definition \ref{defn: the embedding complex}.
 We define two semi-simplicial spaces $K^{\partial}_{\bullet}(M)_{p,q}$ and $\bar{K}^{\partial}_{\bullet}(M, a)_{p,q}$ below. 
\begin{definition} \label{defn: the embedding complex (semi simp)}  
The semi-simplicial space $K^{\partial}_{\bullet}(M)_{p,q}$ is defined as follows:
\begin{enumerate}  \itemsep.2cm
\item[(i)] 
The space of $0$-simplices $K^{\partial}_{0}(M, a)_{p,q}$ is defined to have the same underlying set as the set of vertices of the simplicial complex $K^{\partial}(M, a)_{p,q}$. 
That is, $K^{\partial}_{0}(M, a)_{p,q}$ is the space of pairs 
 $(t, \phi)$, where $t \in \R$ and 
$\phi: \widehat{V}_{p,q} \longrightarrow M$ is an embedding which satisfies the same condition as in part (i) of Definition \ref{defn: the embedding complex}.
\item[(ii)] The space of $l$-simplices, $K^{\partial}_{l}(M, a) \subset (K^{\partial}_{0}(M, a))^{l+1}$ consists ordered $(l+1)$-tuples, 
$((t_{0}, \phi_{0}), \dots, (t_{l}, \phi_{l})),$
such that $t_{0} < \cdots < t_{l}$ and $\phi_{i}(B_{p,q})\cap\phi_{j}(B_{p,q}) = \emptyset$ when $i \neq j$.
\end{enumerate}
The spaces $K^{\partial}_{l}(M, a)\subset (\R\times
  \Emb(\widehat{V}_{p,q}, M))^{l+1}$ are topologized using the $C^{\infty}$-topology on
the space of embeddings. 
The assignments $[l] \mapsto K^{\partial}_{l}(M, a)$ define a semi-simplicial space denoted by 
$K^{\partial}_{\bullet}(M, a)_{p,q}$.
The $i$th face map $d_{i}: K^{\partial}_{l}(M, a) \longrightarrow K^{\partial}_{l-1}(M, a)$ is given by forgetting the $i$th entry in the $l$-tuple $((t_{0}, \phi_{0}), \dots, (t_{l}, \phi_{l}))$. 

Finally, 
$\bar{K}^{\partial}_{\bullet}(M, a)_{p,q} \subset K^{\partial}_{\bullet}(M, a)_{p,q}$ is defined to be the sub-semi-simplicial space consisting of all simplices $((\phi_{0}, t_{0}), \dots, (\phi_{l}, t_{l})) \in K^{\partial}_{l}(M, a)$ such that $\phi_{i}(\widehat{V}_{p,q}) \cap \phi_{j}(\widehat{V}_{p,q}) = \emptyset$ whenever $i \neq j$.
\end{definition} 
As in Remark \ref{remark: drop a from notation}, when working with $K^{\partial}_{\bullet}(M, a)_{p,q}$ and $\bar{K}^{\partial}_{\bullet}(M, a)_{p,q}$ we will drop the embedding $a$ from the notation and write,
$\bar{K}^{\partial}_{\bullet}(M)_{p,q} := \bar{K}^{\partial}_{\bullet}(M, a)_{p,q},$ and $K^{\partial}_{\bullet}(M)_{p,q} := K^{\partial}_{\bullet}(M, a)_{p,q}.$
We will ultimately need to use the fact that the geometric realizations $|\bar{K}^{\partial}_{\bullet}(M, a)_{p,q}|$ and $|K^{\partial}_{\bullet}(M, a)_{p,q}|$ are highly connected. 
We prove that $|\bar{K}^{\partial}_{\bullet}(M, a)_{p,q}|$ and $|K^{\partial}_{\bullet}(M, a)_{p,q}|$ are highly connected by comparing them to the simplicial complex $K^{\partial}(M)_{p,q}$. 
\begin{corollary} \label{corollary: high connectivity of main flag complex}
Let $(M; \partial_{0}M, \partial_{1}M)$ be an $m$-dimensional manifold triad with $\partial_{0}M \neq \emptyset$. 
Let $p, q \in \N$ be such that $p+q+1 = m$ and suppose that 
the inequalities (\ref{equation: main inequalities restated}) are satisfied. 
Let $d = d(\pi_{q}(S^{p}))$ be the generating set rank and suppose that $r_{p,q}(M) \geq g$.
Then the geometric realization $|\bar{K}^{\partial}_{\bullet}(M)_{p,q}|$ is $\tfrac{1}{2}(g-4-d)$-connected.
\end{corollary}

\begin{proof}[Proof of Corollary \ref{corollary: high connectivity of main flag complex}]
The result is obtained from Theorem \ref{theorem: high-connectivity} by assembling several results from \cite{GRW 14}.
Let $g \in \Z_{\geq 0}$ be given.  
Consider the discretization $K^{\partial}_{\bullet}(M)^{\delta}_{p,q}$. 
This is the semi-simplicial space obtained by defining each $K^{\partial}_{l}(M)^{\delta}_{p,q}$ to have the same underlying set as $K^{\partial}_{l}(M)_{p,q}$, but topologized in the discrete topology. 
The correspondence, 
$
((\phi_{0}, t_{0}), \dots, (\phi_{l}, t_{l})) \mapsto \{(\phi_{0}, t_{0}), \dots, (\phi_{l}, t_{l})\},
$
induces a map 
$
|K^{\partial}_{\bullet}(M)^{\delta}_{p,q}| \longrightarrow |K^{\partial}(M)_{p,q}|
$
which is easily seen to be a homeomorphism. 
Indeed, if $\sigma = \{(\phi_{0}, t_{0}), \dots, (\phi_{l}, t_{l})\} \leq K^{\partial}(M)_{p,q}$ is a simplex, then $t_{i} \neq t_{j}$ if $i\neq j$, and so there is exactly one $l$-simplex in $K^{\partial}_{\bullet}(M)^{\delta}_{p,q}$ with underlying set equal to $\sigma$.
It follows from this that $|K^{\partial}_{\bullet}(M)^{\delta}_{p,q}|$ is $\tfrac{1}{2}(g-4-d)$-connected. 
Next, we compare the connectivity of $|K^{\partial}_{\bullet}(M)^{\delta}_{p,q}|$ to $|K^{\partial}_{\bullet}(M)_{p,q}|$. 
By replicating the same argument used in the proof of \cite[Theorem 5.5]{GRW 14} it follows that the connectivity of $|K^{\partial}_{\bullet}(M)_{p,q}|$ is bounded below by the connectivity of $|K^{\partial}_{\bullet}(M)^{\delta}_{p,q}|$, and thus $|K^{\partial}_{\bullet}(M)_{p,q}|$ is $\tfrac{1}{2}(g-4-d)$-connected as well. 
We remark that this proof from \cite{GRW 14} uses the fact that $lCM(K^{\partial}(M)_{p,q}) \geq \tfrac{1}{2}(g-1-d)$ (and this is established in Theorem \ref{theorem: high-connectivity}).
Finally, to finish the proof of the lemma we observe that the inclusion $K^{\partial}_{\bullet}(M)_{p,q} \hookrightarrow \bar{K}^{\partial}_{\bullet}(M)_{p,q}$ is a levelwise weak homotopy equivalence. 
This fact is proven by employing the same argument used in the proof of \cite[Corollary 5.8]{GRW 14}.
This completes the proof of the lemma. 
\end{proof}

High-connectivity of the space $|\bar{K}^{\partial}_{\bullet}(M)_{p,q}|$ is the main ingredient needed for the proof of Theorem \ref{theorem: main homological stability theorem}, which we carry out in Section \ref{section: homological stability}. 
The main ingredient in proving Corollary \ref{corollary: high connectivity of main flag complex} was Theorem \ref{theorem: high-connectivity}. 
The next three sections then are geared toward developing all of the technical tools needed to prove Theorem \ref{theorem: high-connectivity}.

\section{The Algebraic Invariants} \label{Section: The Algebraic Invariants}
\subsection{Invariants of manifold pairs} \label{Subsection: invariants mfd boundary}
For what follows, let $M$ be a compact manifold of dimension $m$ with non-empty boundary. 
Let $A \subset \partial M$ be a submanifold of dimension $m-1$. 
We will keep $M$ and $A$ fixed throughout the entire section. 
Let $p$ and $q$ be positive integers with $p + q +1 = m $ and suppose further that the following inequalities are satisfied:
\begin{equation} \label{equation: inequalities mfd with boundary}
|q - p| < \min\{p, q\} - 2, \quad |q - p|  < \min\{\kappa(A), \kappa(M, A)\} - 1,
\end{equation}
where recall $\kappa(A)$ and $\kappa(M, A)$ denote the degrees of connectivity of $A$ and $(M, A)$ respectively. 
We will consider the homotopy groups $\pi_{q}(A)$ and $\pi_{p+1}(M, A)$.
We will need to be able to represent elements of these homotopy groups by smooth embeddings. 
The next lemma follows by assembling several results from \cite{H 69} and \cite{W 63};
part (i) of the lemma follows by combining the results \cite[Corollary 1.1]{H 69} and \cite[Corollary 2.1]{H 69} and
part (ii) follows by combining the results \cite[Lemma 1]{W 63} and \cite[Proposition 1]{W 63}.

\begin{lemma} \label{lemma: basic embedding lemma}
Let $M$ be a manifold of dimension $m$ with non-empty boundary and let
$A \subset \partial M$ be a submanifold of dimension $m-1$.
Let $p$ and $q$ be positive integers such that $p + q + 1 = m$ and suppose that (\ref{equation: inequalities mfd with boundary}) holds. 
We may then draw the following conclusions about the homotopy groups $\pi_{p+1}(M, A)$ and $\pi_{q}(A)$:
\begin{enumerate} \itemsep.2cm
\item[(i)] Any element of $\pi_{p+1}(M, A)$ can be represented by an embedding $(D^{p+1}, S^{p}) \longrightarrow (M, A)$, unique up to isotopy.
\item[(ii)] Any element of $\pi_{q}(A)$ can be represented by an embedding $S^{q} \longrightarrow A$, unique up to regular homotopy through immersions.
\end{enumerate}
\end{lemma}
We first define a bilinear map 
\begin{equation} \label{equation: boundary tau}
\tau^{\partial}_{p,q}: \pi_{p+1}(M, A)\otimes\pi_{q}(S^{p}) \longrightarrow \pi_{q}(A), \quad ([f], [\phi]) \; \mapsto \; [(f|_{\partial D^{p+1}})\circ \phi],
\end{equation} 
where $f: (D^{p+1}, \partial D^{p+1}) \longrightarrow (M, A)$ represents a class in $\pi_{p+1}(M, A)$ and $\phi: S^{q} \longrightarrow S^{p}$ represents a class in $\pi_{q}(S^{p})$.
By the inequalities of (\ref{equation: inequalities mfd with boundary}), it follows from the \textit{Freudenthal suspension theorem} that the suspension homomorphism $\Sigma: \pi_{q-1}(S^{p-1}) \longrightarrow \pi_{q}(S^{p})$ is surjective. 
For this reason the formula in (\ref{equation: boundary tau}) is indeed bilinear and thus $\tau^{\partial}_{p, q}$ is well defined. 

We then have a bilinear intersection pairing
\begin{equation} \label{equation: lambda intersection}
\lambda^{\partial}_{p, q}: \pi_{p+1}(M, A)\otimes\pi_{q}(A) \longrightarrow \Z, 
\end{equation}
which is defined by sending a pair $([f], [g]) \in \pi_{p+1}(M, A)\otimes\pi_{q}(A)$ to the oriented algebraic intersection number associated to the maps 
$
f|_{\partial D^{p+1}}: S^{p} \longrightarrow A$ and $g: S^{q} \longrightarrow A,$ 
which we may assume are embeddings by Lemma \ref{lemma: basic embedding lemma}. 

We have a $(-1)^{q}$-symmetric bilinear pairing
\begin{equation} \label{equation: symmetric intersection form}
\mu_{q}: \pi_{q}(A)\otimes \pi_{q}(A) \longrightarrow \pi_{q}(S^{p})
\end{equation}
defined in the same way as in \cite[Construction 3.1]{P 15a}. 
We refer the reader there for the definition. 

The next proposition shows how the maps $\tau^{\partial}_{p, q}$, $\lambda^{\partial}_{p, q}$, and $\mu_{q}$ are related to each other. 
The proof follows from \cite[Proposition 3.4]{P 15a}.
\begin{proposition} \label{proposition: lambda-mu linearity}
Let $x \in \pi_{p}(M, A)$, $y \in \pi_{q}(A)$ and $z \in \pi_{q}(S^{p})$. 
Then the equation
$$\mu_{q}(\tau^{\partial}_{p, q}(y, z), x) = \lambda^{\partial}_{p, q}(y, x)\cdot z$$
is satisfied.
\end{proposition}
\begin{remark}
We remark that the formula in Proposition \ref{proposition: lambda-mu linearity} makes sense even in the case when $q < p$ and thus $\pi_{q}(S^{p}) = 0$. 
\end{remark}

We will also need to consider a function 
\begin{equation}
\alpha_{q}: \pi_{q}(A) \longrightarrow \pi_{q-1}(SO_{p})
\end{equation}
defined by sending $x \in \pi_{q}(A)$ to the element in $\pi_{q-1}(SO_{p})$ which classifies the normal bundle associated to an embedding $S^{q} \rightarrow A$ which represents $x$. 

The map $\alpha_{q}$ is not in general a homomorphism. 
As will be seen in Proposition \ref{proposition: additivity defect}, the bilinear form $\mu_{q}$ measures the failure of $\alpha_{q}$ to preserve additivity. 
In order to describe the relationship between $\alpha_{q}$ and $\mu_{q}$, we must define some auxiliary homomorphisms. 
Let 
$\bar{\pi}_{q}: \pi_{q-1}(SO_{p}) \longrightarrow \pi_{q}(S^{p})$
be the map given by the composition $\pi_{q-1}(SO_{p}) \longrightarrow \pi_{q-1}(S^{p-1}) \longrightarrow \pi_{q}(S^{p})$, where the first map is induced by the bundle projection $SO_{p} \longrightarrow SO_{p}/SO_{p-1} \cong S^{p-1}$, and the second is the suspension homomorphism. 
Let 
$d_{q}: \pi_{q}(S^{p}) \longrightarrow \pi_{q-1}(SO_{p})$
be the boundary homomorphism associated to the fibre sequence 
$SO_{p} \longrightarrow SO_{p+1} \longrightarrow S^{p}.$

\begin{proposition} \label{proposition: additivity defect}
The following equations are satisfied for all $x, y \in \pi_{q}(A)$:
$$
\begin{aligned}
\alpha_{q}(x + y) \; &= \; \alpha_{q}(x) + \alpha_{q}(y) + d_{q}(\mu_{t}(x, y)), \\
\mu_{q}(x, x) \; &= \; \bar{\pi}_{q}(\alpha_{q}(x)).
\end{aligned}
$$
\end{proposition}

We will also need the following proposition which describes how $\alpha_{q}$ and $\tau^{\partial}_{p, q}$ are related. 
\begin{proposition}
For all $(x, z) \in \pi_{p+1}(M, A)\times\pi_{q}(S^{p})$, we have $\alpha_{q}(\tau^{\partial}_{p,q}(x, z)) = 0$. 
\end{proposition}
\begin{proof}
Let $\partial_{p+1}: \pi_{p+1}(M, A) \longrightarrow \pi_{p}(A)$ denote the boundary map. 
In \cite[Page 9]{P 15a} a bilinear map 
$$F_{p,q}: \pi_{p-1}(SO_{q})\otimes\pi_{q}(S^{p}) \longrightarrow \pi_{q-1}(SO_{p})$$ 
is defined. 
By \cite[Proposition 3.8]{P 15a} (combined with the definition of the map $\tau_{p,q}^{\partial}$) it follows that 
$$
\alpha_{q}(\tau^{\partial}_{p,q}(x, z)) = F_{p,q}(\alpha_{p}(\partial_{p+1}(x)), \; z)
$$
for all $(x, z) \in \pi_{p+1}(M, A)\times\pi_{q}(S^{p})$. 
Let $f: (D^{p+1}, \partial D^{p+1}) \longrightarrow (M, A)$ be an embedding that represents the class $x \in \pi_{p+1}(M, A)$.
Notice that the normal bundle of $f(D^{p+1})$ is automatically trivial since the disk is contractible. 
It follows that $f|_{\partial D^{p+1}}(\partial D^{p+1}) \subset A$ has trivial normal bundle as well, and thus it follows that $\alpha_{p}(\partial_{p+1}(x)) = 0$ since $f|_{\partial D^{p+1}}$ represents the class $\partial_{p+1}x \in \pi_{p}(A)$. 
Since $F_{p,q}$ is bilinear, it follows that 
$$ \alpha_{q}(\tau^{\partial}_{p,q}(x, z)) \; = \; F_{p,q}(\alpha_{p}(\partial_{p+1}(x)), \; z) = \; F_{p,q}(\alpha_{p}(0), \; z) \; = \;  0$$ 
for all $x$ and $z$. 
This concludes the proof of the proposition.
\end{proof}

Using the maps defined above we will work with the algebraic structure defined by the six-tuple 
\begin{equation} \label{equation: relative Wall-form}
\left(\pi_{p+1}(M, A), \; \pi_{q}(A), \; \tau^{\partial}_{p, q}, \; \lambda^{\partial}_{p, q}, \; \mu_{q}, \; \alpha_{q} \right). 
\end{equation}
We refer to this structure as the \textit{Wall form} associated to the pair $(M, A)$. 
We summarize the salient properties of (\ref{equation: relative Wall-form}) in the following lemma. 
This lemma should be compared to \cite[Lemma 3.9]{P 15a}.
\begin{lemma} \label{lemma: wall form salient properties}
Let $(M, A)$, $p$, and $q$ be exactly as above. The maps 
\begin{itemize} \itemsep.1cm
\item $\tau^{\partial}_{p,q}: \pi_{p+1}(M, A)\otimes\pi_{q}(S^{p}) \longrightarrow \pi_{q}(A)$,
\item $\lambda^{\partial}_{p, q}: \pi_{p+1}(M, A)\otimes\pi_{q}(A) \longrightarrow \Z$,
\item $\mu_{q}: \pi_{q}(A)\otimes\pi_{q}(A) \longrightarrow \pi_{q}(S^{p})$,
\item $\alpha_{q}: \pi_{q}(A) \longrightarrow \pi_{q-1}(SO_{p})$,
\end{itemize}
satisfy the following conditions. For all $x,  x' \in \pi_{p+1}(M, A)$, $y,  y' \in \pi_{q}(A)$ and $z \in \pi_{q}(S^{p})$ we have:
\begin{enumerate} \itemsep.1cm
\item[(i)] 
$\lambda^{\partial}_{p, q}(x, \; \tau^{\partial}_{p,q}(x', z)) = 0,$
\item[(ii)] 
$\mu_{q}(\tau^{\partial}_{p,q}(x,  z), \; y) \; = \; \lambda^{\partial}_{p,q}(x, \; y)\cdot z,$
\item[(iii)]
$\alpha_{q}(y + y') \; = \; \alpha_{q}(y) + \alpha_{q}(y') + d_{q}(\mu_{q}(y, y')),$
\item[(iv)]
$\mu_{q}(y, y) = \bar{\pi}_{q}(\alpha_{q}(y)),$
\item[(v)] 
$\alpha_{q}(\tau^{\partial}_{p,q}(x, z)) = 0$.
\end{enumerate}
\end{lemma}

\subsection{Modifying Intersections} \label{subsection: modifying intersections}
We will ultimately need to use $\lambda^{\partial}_{p, q}$ and $\mu_{q}$ to study intersections of embedded submanifolds.
Let $(M, A)$, $p$ and $q$ be exactly as in the previous section.
For embeddings 
$f: (D^{p+1}, S^{p}) \longrightarrow (M, A)$ and $g: S^{q} \longrightarrow A,$
the integer $\lambda_{p,q}([f],\; [g])$ is equal to the signed intersection number of $f(S^{p})$ and $g(S^{q})$ in $A$. 
Since $A$ is simply-connected, by application of the \textit{Whitney trick} \cite[Theorem 6.6]{M 65}, one can deform $f$ through
a smooth isotopy to a new embedding $f': (D^{p+1}, S^{p}) \rightarrow (M, A)$ such that $f'(S^{p})$ and $g(S^{p})$ intersect transversally in $A$ at exactly $|\lambda^{\partial}_{p,q}(x,y)|$-many points, all with positive orientation.
We now consider embeddings $f, g: S^{q} \rightarrow A$ whose images intersect transversally. 
The intersection $f(S^{q})\cap g(S^{q})$ is generically a $(q - p)$-dimensional closed manifold. 
We will need a higher dimensional analogue of the Whitney trick that applies to the intersection of such embeddings.
The first proposition below follows from \cite{We 67} and \cite{HQ 74}.
\begin{theorem} {\rm (Wells, \cite{We 67})} \label{prop: Generalized Whitney Trick} 
Let
$f, g: S^{q} \longrightarrow A$ be
embeddings. Then there
exists an isotopy $\Psi_{t}: S^{q} \rightarrow A$ with
 $\Psi_{0} = g$
and $\Psi_{1}(S^{q})\cap f(S^{q}) = \emptyset$, if and only if $\mu_{q}([g], [f]) = 0$. 
\end{theorem}

We will also need a technique for manipulating the intersections of embeddings
$(D^{p+1}, \partial D^{p+1}) \longrightarrow (M, A).$
The following theorem is a special case of Theorem \ref{theorem: boundary disjunction}. 
\begin{theorem} \label{theorem: relative higher Whitney trick}
Let $f, g: (D^{p+1}, S^{p}) \longrightarrow (M, A)$ be embeddings.
Then there is an isotopy of embeddings $\Psi_{t}: (D^{p+1}, S^{p}) \longrightarrow (M, A)$ such that $\Psi_{0} = f$ and $\Psi_{1}(D^{p+1})\cap g(D^{p+1}) = \emptyset$.
\end{theorem}
\begin{remark}
We emphasize that in Theorem \ref{theorem: relative higher Whitney trick} the embeddings $f$ and $g$ are completely arbitrary; the theorem holds for any two such embeddings so long as $(M, A)$, $p$, $q$ satisfy (\ref{equation: inequalities mfd with boundary}). 
Furthermore, we emphasize that the restriction $\Psi_{t}|_{S^{p}}$ is not in general the constant isotopy.
The theorem would not be true if we insisted on keeping the restriction $\Psi_{t}|_{S^{p}}$ fixed for all $t$. 
\end{remark}

We will need to apply the above theorems inductively. 
For the statement of the next two results, 
let $(M, A)$ and $p$ and $q$ be exactly as in the statements of the previous two theorems. 
The following corollary is proven in exactly the same way as \cite[Corollary 7.5]{P 15b} using Theorem \ref{prop: Generalized Whitney Trick}. 
\begin{corollary} \label{corollary: inductive whitney trick}
Let 
$
f_{0}, \dots, f_{m}: S^{q} \longrightarrow A
$
be a collection of embeddings such that: 
\begin{enumerate}
\item[(i)] $\mu_{q}(f_{0}, f_{i}) = 0$ for all $i = 1, \dots, m$;
\item[(ii)] the collection of embeddings $f_{1}, \dots, f_{m}$ is pairwise transverse.
\end{enumerate}
Then there exists an isotopy $\Psi_{t}: S^{q} \longrightarrow A$ with $t \in [0,1]$ and $\Psi_{0} = f_{0}$, such that $\Psi_{1}(S^{q})\cap f_{i}(S^{q}) = \emptyset$ for $i = 1, \dots, m$.
\end{corollary}
The next corollary is proven in the same way as \cite[Proposition 6.8]{BP 15} using Theorem \ref{theorem: relative higher Whitney trick}. 
\begin{corollary} \label{corollary: inductive relative whitney trick}
Let $g_{0}, \dots, g_{k}: (D^{p+1}, \partial D^{p+1}) \longrightarrow (M, A)$ be a collection of embeddings such that the collection of submanifolds, 
$g_{1}(D^{p+1}), \dots, g_{k}(D^{p+1}) \subset M,$ 
is pairwise transverse. 
Then there exists an isotopy 
$$\Psi_{t}: (D^{p+1}, \partial D^{p+1}) \longrightarrow (M, A), \quad t \in [0,1],$$ 
with $\Psi_{0} = g_{0}$, such that $\Psi_{1}(D^{p+1})\cap g_{i}(D^{p+1}) = \emptyset$ for $i = 1, \dots, k$.
\end{corollary}

\section{Wall Forms} \label{section: Algebra}
We now formalize the algebraic structure studied in Section \ref{Section: The Algebraic Invariants}.
Much of this section is a recollection of definitions and results from \cite[Section 5]{P 15a}. 
We begin by introducing the category underlying our main construction. 
\begin{definition}
Fix a finitely generated Abelian group $H$. An object $\mb{M}$ in the category $\Ab^{2}_{H}$ is defined to be a pair of abelian groups $(\mb{M}_{-}, \mb{M}_{+})$ equipped with a bilinear map, 
$\tau: \mb{M}_{-}\otimes H \longrightarrow \mb{M}_{+}.$
A morphism $f: \mb{M} \rightarrow \mb{N}$ of $H$-pairs 
 is defined to be a pair of group homomorphisms 
 $f_{-}: \mb{M}_{-} \longrightarrow \mb{N}_{-}$ and $f_{+}: \mb{M}_{+} \longrightarrow \mb{N}_{+}$ 
 that satisfy,
 $f_{+}\circ\tau_{\mb{M}} \; = \; \tau_{\mb{N}}\circ(f_{-}\otimes\textstyle{\Id_{\mb{M}}}).$
 We will refer to morphisms in $\Ab^{2}_{H}$ as \textit{$H$-maps}.
\end{definition}

We build on the above definition as follows. 
Fix once and for all a finitely generated abelian group $H$. 
All of our constructions will take place in the category $\Ab^{2}_{H}$. 
Let $\mb{G}$ be an abelian $H$-pair, equipped with homomorphisms
$\partial: H \longrightarrow \mb{G}_{+}$ and $\pi: \mb{G}_{+} \longrightarrow H.$
Then, let $\epsilon = \pm 1$.
We call such a $4$-tuple $(\mb{G}, \partial, \pi, \epsilon)$ a \textit{form-parameter}. 
Fix a form-parameter $(\mb{G}, \partial, \pi, \epsilon)$ 
and let $\mb{M}$ be a finitely generated $H$-pair.
Consider the following data:
\begin{itemize} 
\item A bilinear map, $\lambda: \mb{M}_{-}\otimes \mb{M}_{+} \longrightarrow \Z$.
\item An $\epsilon$-symmetric bilinear form, $\mu: \mb{M}_{+}\otimes\mb{M}_{+} \longrightarrow H$.
\item Functions, $\alpha_{\pm}: \mb{M}_{\pm} \longrightarrow \mb{G}_{\pm}$.
\end{itemize}
Our main definition is given below.
\begin{definition} \label{defn: Wall-form} The $5$-tuple $(\mb{M}, \lambda, \mu, \alpha)$ is said to be a \textit{Wall form} with parameter $(\mb{G}, \partial, \pi, \epsilon)$ if the following conditions are satisfied
for all $x, \; x' \in \mb{M}_{-}$, \; $y, \; y' \in \mb{M}_{+}$, and $h \in H$:
\begin{enumerate} \itemsep.1cm
\item[(i)] $\lambda(x, \tau_{\mb{M}}(x', h)) = 0$, 
\item[(ii)] 
$\mu(\tau_{\mb{M}}(x, h), y) \; = \; \lambda(x, y)\cdot h, \; $
\item[(iii)]
$\alpha_{-}(x + x') = \alpha_{-}(x) + \alpha_{-}(x'),$
\item[(iv)] 
$\alpha_{+}(y + y') = \alpha_{+}(y) + \alpha_{+}(y') + \partial(\mu(y, y')), $
\item[(v)]
$\mu(y, y) \; = \; \pi(\alpha_{+}(y)),$
\item[(vi)] 
$\alpha_{+}(\tau_{\mb{M}}(x, h)) \; = \; \tau_{\mb{G}}(\alpha_{-}(x), h).$
\end{enumerate}
The Wall form $(\mb{M}, \lambda, \mu, \alpha)$ is said to be \textit{reduced} if $\alpha_{-}$ is identically zero. 
In the case of a reduced Wall form, condition (vi) then translates to $\alpha_{+}(\tau_{\mb{M}}(x, h)) = 0$ for all $x \in \mb{M}_{-}$ and $h \in H$. 
A \textit{morphism} between Wall forms (with the same form-parameter) is an  $H$-map $f: \mb{M} \longrightarrow \mb{N}$ 
 that preserves all values of $\lambda$, $\mu$, and $\alpha$.  
\end{definition}

We will often denote a Wall form by its underlying $H$-pair, i.e. $\mb{M} := (\mb{M}, \lambda, \mu, \alpha)$.
We will need notation for orthogonal complements. 
\begin{definition}
Let $\mb{N} \leq \mb{M}$ be Wall forms. 
We define a new sub-Wall form $\mb{N}^{\perp} \leq \mb{M}$ by setting:
$$
\begin{aligned}
\mb{N}^{\perp}_{-} \; &:= \; \{\; x \in \mb{M}_{-} \; | \; \lambda(x, w) = 0 \; \text{for all $w \in \mb{N}_{+}$} \},\\
\mb{N}^{\perp}_{+} \; &:= \: \{\; y \in \mb{M}_{+} \; | \; \lambda(v, y) = 0 \; \text{and} \; \mu(y, w) = 0 \; \text{for all $v \in \mb{N}_{-}$,  $w \in \mb{N}_{+}$}\}.
\end{aligned}
$$
It can be easily checked that $\tau(\mb{N}^{\perp}_{-}\otimes H) \leq \mb{N}^{\perp}_{+}$ and thus $\mb{N}^{\perp}$ actually is a sub-$H$-pair of $\mb{M}$. 
We call $\mb{N}^{\perp}$ the \textit{orthogonal complement} to $\mb{N}$ in $\mb{M}$. 
Two sub-Wall forms $\mb{N}, \mb{N}' \leq \mb{M}$ are said to be \textit{orthogonal} if $\mb{N}\cap\mb{N}' = \mb{0}$, $\mb{N} \leq (\mb{N}')^{\perp}$  and $\mb{N}' \leq \mb{N}^{\perp}$.
\end{definition}

We will need to use the simplicial complex from \cite[Definition 4.13]{P 15a}.
For this we must recall the definition of the \textit{standard Wall form}.
This requires a few steps. 
Fix a finitely generated Abelian group $H$. 
We define an $H$-pair $\mb{W} \in \Ob(\Ab_{H}^{2})$ by setting, 
$\mb{W}_{-} = \Z$ and $\mb{W}_{+} = \Z\oplus H.$
The map $\tau: \mb{W}_{-}\otimes H \longrightarrow \mb{W}_{+}$ is defined by the formula,
$$
\tau(t\otimes h) \; = \; (0, \; t\cdot h) \; \in \; \Z\oplus H = \mb{W}_{+}.
$$
For $g \in \N$, we denote by $\mb{W}^{g}$ the $g$-fold direct-sum $\mb{W}^{\oplus g}$. 
We let $\mb{W}$ denote the $H$-pair $\mb{W}^{1}$, and $\mb{W}^{0}$ is understood to be the trivial $H$-pair. 
Fix elements 
$a \in \mb{W}_{-}$ and $b \in \mb{W}_{+}$
which correspond to $1 \in \Z$ and $(1, 0) \in \Z\oplus H$ respectively. 
For $g \in \N$, we denote by 
$a_{i} \in \mb{W}^{g}_{-}$ and $b_{i} \in \mb{W}^{g}_{+}$ for $i = 1, \dots, g,$
the elements  that correspond to the elements $a$ and $b$ coming from the $i$th direct-summand of $\mb{W}$ in $\mb{W}^{g}$.
Now fix a form-parameter $(\mb{G}, \partial, \pi, \epsilon)$. 
We endow $\mb{W}^{g}$ with the structure of a Wall form with parameter $(\mb{G}, \partial, \pi, \epsilon)$ by setting:
\begin{equation} \label{eq: standard Wall-form maps}
\lambda(a_{i}, b_{j}) = \delta_{i,j}, \quad \mu(b_{i}, b_{j}) = 0, \quad \alpha_{+}(b_{i}) = 0, \quad \alpha_{-}(a_{i}) = 0 \quad \text{for $i, j = 1, \dots g$.}
\end{equation}
These values together with the conditions imposed from Definition \ref{defn: Wall-form} determine the maps $\lambda, \mu,$ and $\alpha$ completely.
The vanishing of $\alpha_{-}(a_{i})$ for all $i$, implies that $\alpha_{-}$ is identically zero,
thus  $(\mb{W}^{g}, \lambda, \mu, \alpha)$
is a reduced Wall form with parameter $(\mb{G}, \partial, \pi, \epsilon)$ (the fact that $\alpha_{+}(b_{i}) = 0$ for all $i$ does not imply that $\alpha_{+} = 0$ however).
We call this the \textit{standard Wall form of rank $g$} with parameter $(\mb{G}, \pi, \partial, \epsilon)$.

We will use standard Wall form $\mb{W}$ to probe other Wall forms. 
The simplicial complex defined in the next Definition is an algebraic analogue of the simplicial complex defined in Section \ref{section: simplicial complexes}. 
\begin{definition} \label{defn: algebraic simplicial complex} For a Wall form $\mb{M}$ let $L(\mb{M})$ be the simplicial complex whose vertices are given by morphisms $f: \mb{W} \rightarrow \mb{M}$. A set of vertices $\{f_{0}, \dots, f_{l}\}$ is an $l$-simplex if the sub Wall forms, $f_{0}(\mb{W}), \dots, f_{l}(\mb{W}) \leq \mb{M}$ are pairwise orthogonal. 
\end{definition}
To state the main theorem regarding the simplicial complex $L(\mb{M})$ we need to introduce a notion of rank for a Wall form. 
The definition below is analogous to the rank $r_{p,q}(\--)$ associated to a manifold triad $(M; \partial_{0}M, \partial_{1}M)$ defined back in Section \ref{section: Preliminary Constructions}. 
\begin{definition} \label{defn: rank of wall form}
For a Wall form $\mb{M}$, the \text{rank} of $\mb{M}$ is defined to be the non-negative integer,
$$r(\mb{M}) := \max\{ \; g \in \N \; | \; \text{there exists a morphism $\mb{W}^{g} \rightarrow \mb{M}$}\}.$$
 \end{definition}
One of the key technical results proven in \cite{P 15a} (see \cite[Theorem 5.1]{P 15a}) is the theorem stated below.
For the statement of this theorem, let $d$ denote the generating set rank $d(H)$, which recall is the quantity, 
$
d(H) \; = \; \min\{ \; k \in \N \; | \; \text{there exists an epimorphism $\Z^{\oplus k} \longrightarrow H$} \; \}.
$
\begin{theorem} \label{thm: high connectivity} Suppose that
$r(\mb{M}) \geq g$. Then $lCM(L(\mb{M})) \geq
\frac{1}{2}(g -1-d)$ and the geometric realization
$|L(\mb{M})|$ is $\frac{1}{2}(g - 4 - d)$ connected.
\end{theorem}

We now show how to use the constructions from Section \ref{Section: The Algebraic Invariants} to associate a Wall form to a pair of manifolds.
Let $M$ be a compact oriented manifold of dimension $m$ with non-empty boundary. 
Let $A \subset \partial M$ be a submanifold of dimension $m-1$. 
Let $p$ and $q$ be positive integers with $p + q + 1 = m$.
Suppose that $p$ and $q$ satisfy the inequalities from
(\ref{equation: inequalities mfd with boundary}).
We set $H = \pi_{q}(S^{p})$ and denote by $\mathcal{W}^{\partial}_{p,q}(M, A)$ the $H$-pair given by setting,
$$
\mathcal{W}^{\partial}_{p,q}(M, A)_{-} := \pi_{p+1}(M, A),  \quad
\mathcal{W}^{\partial}_{p,q}(M, A)_{+} := \pi_{q}(A), 
$$
and then by setting $\tau$ equal to the bilinear map, 
$$\tau:= \tau^{\partial}_{p,q}:  \pi_{p+1}(M, A)\otimes\pi_{q}(S^{p}) \longrightarrow \pi_{q}(A),$$
from (\ref{equation: boundary tau}).
We need to define a suitable form-parameter. 
Let $\mathbf{G}_{p,q}$ denote the abelian group $\pi_{p-1}(SO_{q})$.
The group $\mathbf{G}_{p,q}$ together with the maps from Proposition \ref{proposition: additivity defect},
$d_{q}: \pi_{q}(S^{p}) \rightarrow \pi_{q-1}(SO_{p})$ and $\bar{\pi}_{q}: \pi_{q-1}(SO_{p}) \rightarrow \pi_{q}(S^{p}),$
make the $4$-tuple
$(\mb{G}_{p, q}, \; d_{q}, \; \pi_{q},\; (-1)^{q})$
into a form parameter.
It follows then directly from Lemma \ref{lemma: wall form salient properties} that the $4$-tuple, 
\begin{equation} \label{eq: Wall-form of M}
(\mathcal{W}^{\partial}_{p, q}(M, A), \; \lambda^{\partial}_{p, q}, \; \mu_{q}, \; \alpha_{q}),
\end{equation}
is a reduced Wall form with form-parameter $(\mb{G}_{p, q}, \; d_{q}, \; \pi_{q}, (-1)^{q})$. 
We call the Wall form of (\ref{eq: Wall-form of M}) the \textit{Wall form of degree $(p, q)$ associated to $(M, A)$}.
This construction should be compared to \cite[Section 4.3]{P 15a}.

We now state two basic propositions that follow directly from the definitions of $\tau^{\partial}_{p,q}$, $\lambda^{\partial}_{p,q}$, $\mu_{q}$, and $\alpha_{q}$.
\begin{proposition}
Let $M$ and $N$ be $m$ dimensional manifolds with non-empty boundary. 
Let $A \subset \partial M$ and $B \subset \partial N$ be submanifolds of dimension $m-1$.  
Let $p, q \in \Z_{\geq 0}$ be chosen with $p + q + 1 = m$ so that the Wall forms $\mathcal{W}^{\partial}_{p,q}(N, B)$ and $\mathcal{W}^{\partial}_{p,q}(M, A)$ are defined. 
Then any embedding 
$\varphi: (N, B) \longrightarrow (M, A)$
induces a unique morphism of Wall forms
$\varphi_{*}: \mathcal{W}^{\partial}_{p,q}(N, B) \longrightarrow \mathcal{W}^{\partial}_{p,q}(M, A).$
\end{proposition}

Let $(M, A)$ and $(N, B)$ be exactly as in the previous proposition. 
Consider the pair $(M\natural N, A\# B)$ obtained by forming the boundary connected sum of $M$ with $N$ along embeddings,
$(M, A) \hookleftarrow (D^{m}_{+}, \partial_{0}D^{m}_{+}) \hookrightarrow (N, B).$ 
\begin{proposition} \label{proposition: direct sum-connect sum}
There is an isomorphism of Wall forms,
$\mathcal{W}^{\partial}_{p,q}(M\natural N, A\# B) \; \cong \; \mathcal{W}^{\partial}_{p, q}(M, A)\oplus\mathcal{W}^{\partial}_{p,q}(N, B).$
\end{proposition}

\section{High Connectivity of $K^{\partial}(M)_{p,q}$} \label{section: high connectivity of K}
Let $(M; \partial_{0}M, \partial_{1}M)$ be an $m$-dimensional manifold triad with $\partial_{0}M \neq \emptyset$. 
Let $p$ and $q$ be positive integers with $p+q+1 = m$ and suppose that the inequalities of (\ref{equation: inequalities mfd with boundary}) are satisfied.
In this section we will prove Theorem \ref{theorem: high-connectivity}
which asserts that $|K^{\partial}(M)_{p,q}|$ is $\tfrac{1}{2}(r_{p,q}(M) - 4 - d)$-connected and that 
$lCM(K^{\partial}(M)_{p,q}) \geq \tfrac{1}{2}(r_{p,q}(M) - 1 - d)$, where $d = d(\pi_{q}(S^{p}))$ is the generating set rank. 
Our strategy is to compare $K^{\partial}(M)_{p,q}$ to the complex $L(\mathcal{W}^{\partial}_{p,q}(M, \partial_{1}M))$ from Definition \ref{defn: algebraic simplicial complex}, associated to the Wall form $\mathcal{W}^{\partial}_{p,q}(M, \partial_{1}M)$.
In view of Theorem \ref{thm: high connectivity}, we will need to construct a simplicial map 
$K^{\partial}(M)_{p,q} \longrightarrow L(\mathcal{W}^{\partial}_{p,q}(M, \partial_{1}M))$
and then prove that it is highly connected. 
The construction of this map and proof of its high-connectivity is carried out over the course of this section, which contains the technical core of the paper. 

\subsection{A simplicial map} \label{subsection: A simplicial map}
Let $(M; \partial_{0}M, \partial_{1}M)$ $m$-dimensional manifold triad with $\partial_{0}M \neq \emptyset$ and let 
$p$ and $q$ be positive integers with $p+q+1 = m$. 
We will construct a simplicial map 
\begin{equation} \label{equation: simplicial comparison map}
F_{p,q}: K^{\partial}(M)_{p,q} \longrightarrow L(\mathcal{W}^{\partial}_{p,q}(M, \partial_{1}M)).
\end{equation}
Recall the pair $(\widehat{V}_{p,q}, \widehat{W}_{p,q})$. 
Recall from Construction \ref{construction: main probing manifold} the core,
$(B_{p, q}, C_{p, q}) \stackrel{\simeq} \hookrightarrow  (\widehat{V}_{p, q}, \widehat{W}_{p,q}).
$
Let $(a_{0}, b_{0}) \in S^{p}\times S^{q}$ be the basepoint used in the construction of $(B_{p, q}, C_{p,q})$ from Construction \ref{construction: main probing manifold}. 
Let $\sigma \in \pi_{p+1}(\widehat{V}_{p, q}, \widehat{W}_{p,q})$ be the class represented by the embedding,
$
(D^{p+1}\times\{b_{0}\}, S^{p}\times\{b_{0}\}) \hookrightarrow (B_{p, q}, C_{p, q}) \hookrightarrow (\widehat{V}_{p, q}, \widehat{W}_{p,q}),
$
and let $\rho \in \pi_{q}(\widehat{W}_{p,q})$ be the class represented by the embedding,
$
\{a_{0}\}\times S^{q} \hookrightarrow C_{p, q} \hookrightarrow \widehat{W}_{p,q}.
$
It follows directly from the construction of $\widehat{V}_{p,q}$ that 
$\lambda_{p, q}(\sigma, \rho) = 1$ and $\alpha_{q}(\rho) = 0.$
Using this observation we may define a morphism of Wall forms
\begin{equation} \label{equation: main morphism of standard Wall form}
T_{p,q}: \mb{W} \; \longrightarrow \; \mathcal{W}^{\partial}_{p,q}(\widehat{V}_{p, q}, \widehat{W}_{p,q}), \quad \quad a \mapsto \sigma, \quad b \mapsto \rho,
\end{equation}
where $a \in \mb{W}_{-}$ and $b \in \mb{W}_{+}$ are the standard generators used in the construction of $\mb{W}$. 
Using the calculations,
$$
\pi_{p+1}(\widehat{V}_{p, q}, \widehat{W}_{p,q}) \cong \Z, \quad \pi_{q}(\widehat{W}_{p,q}) \cong \pi_{q}(S^{p})\oplus \Z,
$$
it follows that this morphism $T_{p,q}$ is an isomorphism of Wall forms. 
Combining this isomorphism together with Proposition \ref{proposition: direct sum-connect sum} we obtain the following result.
\begin{proposition}
For all $g \in \Z_{\geq 0}$, there is an isomorphism of Wall forms,
$\mb{W}^{g} \cong \mathcal{W}^{\partial}_{p,q}(V^{g}_{p, q}, W^{g}_{p,q})$.
\end{proposition}

We are now ready to define the simplicial map $F_{p,q}$ from (\ref{equation: simplicial comparison map}). 
Let $\phi \in K^{\partial}(M)_{p,q}$ be a vertex. 
We define $F_{p,q}(\phi) \in L(\mathcal{W}^{\partial}_{p,q}(M, \partial_{1}M))$ to be the morphism of Wall forms given by the composite,
\begin{equation}
\xymatrix{
\mb{W} \ar[rr]^{T_{p,q}\ \ \ \ \ } && \mathcal{W}^{\partial}_{p,q}(\widehat{V}_{p,q}, \widehat{W}_{p,q}) \ar[rr]^{\phi_{*}} && \mathcal{W}^{\partial}_{p,q}(M, \partial_{1}M),
}
\end{equation}
where the second map $\phi_{*}$ is the morphism of Wall forms induced by the embedding $\phi$. 
Using this morphism $T_{p,q}$ together with Proposition \ref{proposition: direct sum-connect sum}, it follows that 
$r(\mathcal{W}^{\partial}_{p,q}(M, \partial_{1}M)) \; \geq \; r_{p,q}(M).$

\subsection{Cohen Macaulay complexes and the link lifting property} \label{subsection: a simplicial technique}

Our next step is to prove that the simplicial map $T_{p,q}$ from (\ref{equation: simplicial comparison map}) is highly connected.
Doing this will require us to import a simplicial technique.
\begin{definition} \label{defn: link lifting property 1}
Let $f: X \longrightarrow Y$ be a simplicial map between two simplicial complexes. 
The map is said to have the \textit{link lifting property} if the following condition holds:
\begin{itemize}
\item Let $y \in Y$ be a vertex and let $A \subset X$ be a set of vertices such that $f(a) \in \lk_{Y}(y)$ for all $a \in A$.
Then there exists $x \in X$ with $f(x) = y$ such that $a \in \lk_{X}(x)$ for all $a \in A$.
\end{itemize}
\end{definition}

In order for the map $f$ to have the link lifting property it is necessary that the above condition be satisfied for any set of vertices $A \subset X$. 
In practice, it can be difficult to verify the above lifting property for totally arbitrary subsets $A \subset X$, and one
may only be able to verify it for sets of vertices subject to some condition. 
We will need to work with a refined version of Definition \ref{defn: link lifting property 1}. 
Recall that for any set $K$, a \textit{symmetric relation} is a subset $\mathcal{R} \subset K\times K$ that is invariant under the ``coordinate permutation'' map $(x, y) \mapsto (y, x)$. 
A subset $C \subset K$ is said to be in \textit{general position with respect to $\mathcal{R}$} if $(x, y) \in \mathcal{R}$ for any two elements $x, y \in C$. 
We will need to consider symmetric relations defined on the set of vertices of a simplicial complex. 
Let $X$ be a simplicial complex and let $\mathcal{R} \subset X\times X$ be a symmetric relation on the vertices of $X$. 
The relation $\mathcal{R}$ is said to be \textit{edge compatible} if for any $1$-simplex $\{x, y\} < X$, the pair $(x, y)$ is an element of $\mathcal{R}$.

\begin{definition} \label{defn: cone lifting property} 
Let $f: X \longrightarrow Y$ be a simplicial map between two simplicial complexes. 
Let $\mathcal{R} \subset X\times X$ be a symmetric relation on the set of vertices of the complex $X$.
The map $f$ is said to have the \textit{link lifting property with respect to $\mathcal{R}$} if the following condition holds:
\begin{itemize}
\item
Let $y \in Y$ be any vertex and let $A \subset X$ be a finite set of vertices in general position with respect to $\mathcal{R}$ such that
 $f(a) \in \lk_{Y}(x)$ for all $a \in A$.
Then given another finite set of vertices $B \subset X$ (not necessarily in general position with respect to $\mathcal{R}$), there exists a vertex $x \in X$ with $f(x) = y$ such that 
$a \in \textstyle{\lk_{X}(x)}$ for all $a \in A$ and $(b, x) \in \mathcal{R}$
for all $b \in B$.
\end{itemize}
\end{definition}
The following lemma below is a restatement of \cite[Lemma 2.3]{P 15b}.
Its proof in \cite{P 15b} abstracts and formalizes the argument used in the proof of \cite[Lemma 5.4]{GRW 14}.
It also uses \cite[Theorem 2.4]{GRW 14} which is a generalization of the ``Coloring Lemma'' of Hatcher and Wahl from \cite{HW 10}.
\begin{lemma} \label{lemma: link lift lemma}
Let $X$ and $Y$ be simplicial complexes and let $f: X \longrightarrow Y$ be a simplicial map.
Let $\mathcal{R} \subset X\times X$ be an edge compatible symmetric relation. 
Suppose that the following conditions are met:
\begin{enumerate} \itemsep.1cm
\item[(i)] $f$ has the \textit{link lifting property} with respect to $\mathcal{R}$;
\item[(ii)] $lCM(Y) \geq n$;
\end{enumerate}
Then the induced map $|f|_{*}: \pi_{j}(|X|) \longrightarrow \pi_{j}(|Y|)$ is injective for all $j \leq n-1$. 
Furthermore, suppose that in addition to properties (i) and (ii) the map $f$ satisfies:
\begin{enumerate}
\item[(iii)] $f(\lk_{X}(\zeta)) \leq \lk_{Y}(f(\zeta))$ for all simplices $\zeta < X$.
\end{enumerate}
Then it follows that $lCM(X) \geq n$. 
\end{lemma}

\subsection{Proof of Theorem \ref{theorem: high-connectivity}} \label{subsection: proof of main high connectivity theorem}
We now give the proof of Theorem \ref{theorem: high-connectivity}. 
We do this by applying Lemma \ref{lemma: link lift lemma} to the simplicial map $F_{p,q}: K^{\partial}(M)_{p,q} \longrightarrow L(\mathcal{W}_{p,q}(M, \partial_{1}M))$.
In order to apply this lemma to $F_{p,q}$, we will need to define a suitable symmetric relation on the vertices of the complex $K^{\partial}(M)_{p,q}$.
\begin{definition} \label{defn: transversality relation}
We define $\mathcal{T} \subset K^{\partial}(M)_{p,q}\times K^{\partial}(M)_{p,q}$ to be the subset consisting of those pairs $(\phi_{1}, \phi_{2})$ such that $\phi_{1}(B_{p,q})$ and $\phi_{2}(B_{p,q})$ are transverse in $M$. 
\end{definition}
Clearly the subset $\mathcal{T} \subset K^{\partial}(M)_{p,q}\times K^{\partial}(M)_{p,q}$ is a symmetric relation on the vertices. 
Furthermore, this relation is edge compatible, i.e.\ if the set $\{\phi_{1}, \phi_{2}\} \leq K^{\partial}(M)_{p,q}$ is a $1$-simplex then the pair $(\phi_{1}, \phi_{2})$ is contained in $\mathcal{T}$.
\begin{proof}[Proof of Theorem \ref{theorem: high-connectivity}] 
Let $r_{p,q}(M) \geq g$ and let $d = d(\pi_{q}(S^{p}))$ be the generating set rank of $\pi_{q}(S^{p})$. 
We will show that 
$|K(M)_{p,q}|$ is $\tfrac{1}{2}(g - 4 - d)$-connected and 
$lCM(K^{\partial}(M)_{p,q}) \geq \tfrac{1}{2}(g - 1 - d).$
Since $r(\mathcal{W}^{\partial}_{p,q}(M)) \geq r_{p,q}(M) \geq g$, Theorem \ref{thm: high connectivity} implies that the space $|L(\mathcal{W}^{\partial}_{p,q}(M, \partial_{1}M))|$ is $\tfrac{1}{2}(g - 4 - d)$-connected and that
$lCM\left(L(\mathcal{W}^{\partial}_{p,q}(M, \partial_{1}M)\right) \geq \tfrac{1}{2}(g - 1 - d).$
The proof of the theorem will follow from Lemma \ref{lemma: link lift lemma} once we verify the following two properties:
\begin{enumerate} \itemsep2pt
\item[(i)] the map $F_{p,q}$ has the link lifting property with respect to $\mathcal{T}$ (see Definition \ref{defn: cone lifting property}),
\item[(ii)] $F_{p,q}(\lk_{K^{\partial}(M)_{p,q}}(\zeta)) \leq \lk_{L(\mathcal{W}^{\partial}_{p,q}(M, \partial_{1}M)}(F_{p,q}(\zeta))$ for any simplex $\zeta \in K^{\partial}(M)_{p,q}$.
\end{enumerate}
We begin by verifying property (i).
Let $f: \mb{W} \longrightarrow \mathcal{W}^{\partial}_{p,q}(M, \partial_{1}M)$ be a morphism of Wall forms, which we consider to be a vertex of $L(\mathcal{W}^{\partial}_{p,q}(M, \partial_{1}M))$. 
Let $\phi_{1} \dots, \phi_{k} \in K^{\partial}(M)_{p,q}$ be a collection of vertices in general position with respect to $\mathcal{T}$, such that $F_{p,q}(\phi_{i}) \in \lk_{L(\mathcal{W}^{\partial}_{p,q}(M, \partial_{1}M))}(f)$ for $i = 1, \dots, k$.
Let $\psi_{1}, \dots, \psi_{m} \in K^{\partial}(M)_{p,q}$ be another arbitrary collection of vertices. 
To show that $F_{p,q}$ has the link lifting property with respect to $\mathcal{T}$, 
we will construct a vertex $\phi \in K^{\partial}(M)_{p,q}$ with $F_{p, q}(\phi) = f$, such that $\phi_{i} \in \lk_{K^{\partial}(M)_{p,q}}(\phi)$ for $i = 1, \dots, k$ and $(\phi_{j}, \phi) \in \mathcal{T}$ for $j = 1, \dots, m$.

Let 
$
\zeta: (D^{p+1}, S^{p}) \longrightarrow (M, \partial_{1}M)$ and $\xi: S^{q} \longrightarrow \partial_{1}M
$
be embeddings that represent the classes
$$f_{-}(a) \in \pi_{p+1}(M, \partial_{1}M) \quad \text{and} \quad f_{+}(b) \in \pi_{q}(\partial_{1}M)$$ 
respectively (it follows from Lemma \ref{lemma: basic embedding lemma} that these classes may be represented by embeddings). 
Since the cores $\phi_{1}(B_{p,q}), \dots, \phi_{k}(B_{p,q})$ are transverse and $F_{p,q}(\phi_{i}) \in \lk_{K^{\partial}(M)_{p,q}}(f)$ for all $i = 1, \dots, k$, we may apply Corollary \ref{corollary: inductive whitney trick} and Corollary \ref{corollary: inductive relative whitney trick} 
to deform the embeddings $\zeta$ and $\xi$ through isotopies to new embeddings $\zeta'$ and $\xi'$ such that the images $\zeta'(D^{p+1})$ and $\xi'(S^{q})$ are disjoint from the cores $\phi_{i}(B_{p,q})$ for $i = 1, \dots, k$. 
Let $\partial_{1}M'$ denote the complement $\partial_{1}M\setminus\left(\cup_{i=1}^{k}\phi_{i}(C_{p,q})\right)$.
Since $\lambda_{p,q}(f_{-}(a), f_{+}(b)) = 1$, we may apply the \textit{Whitney trick} to deform $\xi': S^{q} \longrightarrow \partial_{1}M'$, through an isotopy of embeddings into $\partial_{1}M'$, to a new embedding $\xi'': S^{q} \longrightarrow \partial_{1}M'$ with the property that $\zeta'(\partial D^{p+1})$ and $\xi''(S^{q})$ intersect transversally at exactly one point in $\partial_{1}M'$. 
By \textit{Thom's transversality theorem} we may further arrange $\zeta'(D^{p+1})$ and $\xi''(S^{q})$ to be transverse to each of the cores $\psi_{1}(B_{p,q}), \dots, \psi_{m}(B_{p,q})$ while keeping them disjoint from $\phi_{1}(B_{p,q}), \dots, \phi_{k}(B_{p,q})$.

It follows by the above construction that the pair of subspaces
$$\left(\zeta'(D^{p+1})\cup\xi''(S^{q}), \; \zeta'(\partial D^{p+1})\cup\xi''(S^{q})\right)$$
is homeomorphic to the pair $(D^{p+1}\vee S^{q},\; S^{p}\vee S^{q})$.
Now, both $\zeta'(S^{p})$ and $\xi''(S^{q})$ have trivial normal bundles in $\partial_{1}M$; the normal bundle of $\xi''(S^{q})$ is trivial because $\alpha_{q}(f_{+}(b)) = 0$ and the normal bundle of $\zeta'(S^{p})$ is trivial because it bounds the disk $(\zeta'(D^{p+1}), \; \zeta'(S^{p})) \subset (M, \partial_{1}M)$ which must have trivial normal bundle since the disk is contractible.
Let $U \subset \partial_{1}M$ be a regular neighborhood of $\zeta'(S^{p})\cup\xi''(S^{q})$ (which is a wedge of a $p$-sphere with a $q$-sphere).
The manifold $U$ is diffeomorphic to the manifold obtained by forming the push-out of the digram 
$$
\xymatrix{
S^{p}\times D^{q}  && D^{p}\times D^{q} \ar[ll]_{i_{p}\times\Id_{D^{q}}} \ar[rr]^{\Id_{D^{p}}\times i_{q}} && D^{p}\times S^{q},
}
$$
where $i_{p}: D^{p} \hookrightarrow S^{p}$ and $i_{q}: D^{q} \hookrightarrow S^{q}$ are embeddings. 
It is easily seen that this push-out is diffeomorphic to $W_{p,q} = S^{p}\times S^{q}\setminus\Int(D^{p+q})$ after smoothing corners, thus we have a diffeomorphism $U \cong W_{p,q}$.
By shrinking $U$ down arbitrarily close to its ``core''  $\zeta'(S^{p})\cup\xi'(S^{q}) \cong S^{p}\vee S^{p}$, we may assume that $U$ is 
disjoint from $\phi_{i}(C_{p,q})$ for all $i = 1, \dots, k$. 

Let $\bar{U} \subset M$ be the submanifold diffeomorphic to $U\times[0,1]$ obtained by adding a collar to $U \subset \partial_{1}M$ in $M$. 
We extend the embedding $\zeta'$ to an embedding 
$
\bar{\zeta}: (D^{p+1}\times D^{q}, \; S^{p}\times D^{q}) \; \longrightarrow (M, \partial_{1}M)
$
such that $\bar{\zeta}(S^{p}\times D^{q}) \subset U$ and $\bar{\zeta}|_{S^{p}\times\{0\}} = \zeta'$.
We 
let $V \subset M$ be the subspace obtained by forming the union of $\bar{U}$ with $\bar{\zeta}(D^{p+1}\times D^{q})$. 
By shrinking $\bar{\zeta}(D^{p+1}\times D^{q})$ down to $\bar{\zeta}(D^{p+1}\times\{0\})$, we may assume that $V$ is again disjoint from $\phi_{i}(B_{p,q})$ for all $i = 1, \dots, k$.
By Proposition \ref{proposition: recognition of handlebody} (proven below) there is a diffeomorphism $V \cong V_{p,q} = D^{p+1}\times S^{q}$.
Furthermore, the boundary of $V$ has the decomposition $\partial V = \partial_{0}V\cup\partial_{1}V$, with
$\partial_{1}V = \partial V\cap\partial_{1}M$ and $\partial_{0}V = \partial V\setminus\Int(\partial_{1}V).$
We have diffeomorphisms
$\partial_{1}V \cong W_{p,q}$ and $\partial_{0}V \cong D^{p+q}.$
Using these identifications 
$$V \cong V_{p,q}, \quad \partial_{0}V \cong D^{p+q}, \quad \text{and} \quad \partial_{1}V \cong W_{p,q},$$ 
we obtain an embedding $(V_{p,q}, W_{p,q}) \hookrightarrow (M, \partial_{1}M)$ with image equal to 
$(V, \partial_{1}V) \subset (M, \partial_{1}M)$.

We then choose an embedding $\gamma: [0, 1] \hookrightarrow \partial_{1}M$, disjoint from $\phi_{i}(B_{p,q})$ for all $i = 1, \dots, k$, and with $\gamma(0) \in \partial_{0,1}V$  and $\gamma(1) \in\partial_{0}M$.
Taking the union of a thickening of this arc with $V$ yields an embedding 
$(\widehat{V}_{p,q}, \widehat{W}_{p,q}) \longrightarrow (M, \partial_{1}M)$
that satisfies condition i. of Definition \ref{defn: the embedding complex}. 
This in turn yields a vertex $\phi \in K^{\partial}(M)_{p,q}$ with $F_{p,q}(\phi) = f$ such that $\phi(B_{p,q})\cap\phi_{i}(B_{p,q}) = \emptyset$ for all $i = 1, \dots, k$. 
It follows that $\phi_{i}$ is contained in the link of $\phi$ for $i = 1, \dots, k$.
By construction, $\phi(B_{p,q})$ is transverse to $\psi_{j}(B_{p,q})$ for all $j = 1, \dots, m$.
This proves that the map $F_{p,q}$ has the link lifting property with respect to $\mathcal{T}$.

By Lemma \ref{lemma: link lift lemma} it follows that the induced map $\pi_{i}(|K^{\partial}(M)_{p,q}|) \longrightarrow \pi_{i}(|L(\mathcal{W}^{\partial}_{p,q}(M, \partial_{1}M)|)$ is injective for all $i < \tfrac{1}{2}(g - 1 - d)$ and thus $|K^{\partial}(M)_{p,q}|)$ is $\tfrac{1}{2}(g - 4 - d)$-connected. 
In order to conclude that $lCM(K^{\partial}(M)_{p,q}) \geq \tfrac{1}{2}(g - 1 - d)$, we need to establish property (ii).
We need to verify that, 
$$F(\textstyle{\lk_{K^{\partial}(M)_{p,q}}}(\zeta)) \leq \textstyle{\lk_{L(\mathcal{W}^{\partial}_{p,q}(M, \partial_{1}M)}}(F_{p,q}(\zeta))$$ 
for any simplex $\zeta \in K^{\partial}(M)_{p,q}.$
This property follows immediately from the fact that if $\phi_{1}, \phi_{2} \in K^{\partial}(M)_{p,q}$ are such that $\phi_{1}(B_{p,q})\cap\phi_{2}(B_{p,q}) = \emptyset$, then the morphisms of Wall forms $F_{p,q}(\phi_{1})$ and $F_{p,q}(\phi_{2})$ are orthogonal.
This concludes the proof of the theorem. 
\end{proof}

There is one claim in the above proof that still needs verification, namely that the manifold $V$ that we constructed is diffeomorphic to $V_{p,q}$.
We now prove that claim.
Pick a base point $(a, b) \in S^{p}\times S^{q}$ such that $(S^{p}\times\{b\})\cup(\{a\}\times S^{q}) \subset W_{p,q} = S^{p}\times S^{q}\setminus\Int(D^{p+q})$.
Let $f: S^{p} \longrightarrow W_{p,q}$ be the embedding given by the chain of inclusions,
$
S^{p} \hookrightarrow S^{p}\times\{b\} \hookrightarrow (S^{p}\times\{b\})\cup(\{a\}\times S^{q}) \hookrightarrow W_{p,q}.
$
This embedding has a trivial normal bundle.
Let $f': S^{p}\times D^{q} \longrightarrow W_{p,q}$ be an embedding with $f'|_{S^{p}\times\{0\}} = f$. 
Finally let 
$$\bar{f}: S^{p}\times D^{q} \longrightarrow W_{p,q}\times[0,1]$$ 
denote the embedding given by, 
$
S^{p}\times D^{q} \stackrel{f'} \longrightarrow  W_{p,q} \hookrightarrow W_{p,q}\times\{1\} \hookrightarrow W_{p,q}\times[0,1].
$
Let $V$ denote the manifold obtained by attaching the handle $D^{p+1}\times D^{q}$ to $W_{p,q}\times[0,1]$ along the embedding $\bar{f}$, i.e.\
$$
V \; = \; \left(W_{p,q}\times[0,1]\right)\bigcup_{\bar{f}}\left(D^{p+1}\times D^{q}\right).
$$
The boundary of $W_{p,q}$ has the decomposition,
$
\partial V = (W_{p,q}\times\{0\})\cup(\partial W_{p,q}\times[0,1])\cup W',
$
where $W'$ is the manifold obtained from $W_{p,q}$ by performing surgery along the embedding $\bar{f}$. 
Consequentially the manifold $W'$ is diffeomorphic to a disk $D^{p+q}$ and so $\partial V$ is diffeomorphic to $W_{p,q}$.
The following proposition was used in the above proof of Theorem \ref{theorem: high-connectivity}. 
\begin{proposition} \label{proposition: recognition of handlebody}
Let $p, q \in \N$ satisfy the inequality $|q - p| < \min\{p, q\} - 2$. 
Then the manifold $V$ constructed above is diffeomorphic to $D^{p+1}\times S^{q} = V_{p,q}$.
\end{proposition}
\begin{proof}
Let $g: S^{q} \longrightarrow V$ be the embedding given by the chain of inclusions,
$$
S^{q} \hookrightarrow \{a\}\times S^{q} \hookrightarrow W_{p,q} \hookrightarrow W_{p,q}\times\{\tfrac{1}{2}\} \hookrightarrow W_{p,q}\times[0,1] \hookrightarrow V.
$$
This embedding has trivial normal bundle and so extends to an embedding, 
$\bar{g}: D^{p+1}\times S^{q} \longrightarrow \Int(V).$ 
Furthermore, $\bar{g}$ induces an isomorphism on homology and thus is a homotopy equivalence since $V$ is simply connected. 
Let $X$ denote the complement $V \setminus \Int(\bar{g}(D^{p+1}\times S^{q}))$. 
The boundary of $X$ decomposes as the disjoint union $\partial X = \partial V \;\sqcup\; \partial\bar{g}(D^{p+1}\times S^{q})$.
By excision we have the isomorphism 
$0 = H_{i}(V, \bar{g}(D^{p+1}\times S^{q})) \; \cong \; H_{i}(X, \partial\bar{g}(D^{p+1}\times S^{q}))$ for all $i \in \Z_{\geq 0}$,
and then by \textit{Lefschetz duality} we obtain 
$$0 \; = \; H_{i}(X, \partial\bar{g}(D^{p+1}\times S^{q})) \; \cong \; H^{i}(X, \partial V) = 0 \quad \text{for all $i \in \Z_{\geq 0}$.}$$ 
Since $X$ and $\partial V$ are simply connected, it follows that $X$ is an $h$-cobordism between $\partial V$ and the manifold $\partial\bar{g}(D^{p+1}\times S^{q})$. 
Since $p+q+1 \geq 6$, it follows by the \textit{H-cobordism theorem} that $X$ is diffeomorphic to the cylinder $\partial V\times[0,1]$. 
By shrinking down this cylinder, it follows that the embedding $\bar{g}: D^{p+1}\times S^{q} \longrightarrow V$ is isotopic to a diffeomorphism $D^{p+1}\times S^{q} \cong V$. 
This completes the proof of the proposition.
\end{proof}
With the above proposition established, the proof of Theorem \ref{theorem: high-connectivity} is now complete. 
By Corollary \ref{corollary: high connectivity of main flag complex}, it now follows that the geometric realization $|\bar{K}^{\partial}(M)_{p,q}|$ is $\tfrac{1}{2}[g-4-d(\pi_{q}(S^{p}))]$-connected whenever $r_{p,q}(M) \geq g$.

\section{Homological Stability} \label{section: homological stability}
With our main technical result Theorem \ref{theorem: high-connectivity} established, in this section we show how this theorem implies our homological stability theorem, Theorem \ref{theorem: main homological stability theorem}.
The constructions and arguments in this section are essentially the same as what was done in \cite[Section 6]{GRW 14} (or the earlier unpublished preprint \cite[Section 5]{GRW 12}) and so we merely provide an outline while referring the reader to those above mentioned papers for details. 
\subsection{A Model for $\BDiff(M, \partial_{0}M)$.} \label{subsection: semi-simplicial resolution}
Let $(M; \partial_{0}M, \partial_{1}M)$ be a compact manifold triad of dimension $m$ with $\partial_{0}M$ and $\partial_{1}M$ non-empty. 
We construct a concrete model for $\BDiff(M, \partial_{0}M)$. 
 Fix once and for all an embedding, 
 $\theta: (\partial_{0}M, \partial_{0,1}M)  \longrightarrow (\R^{\infty}_{+}, \; \partial\R^{\infty}_{+})$
 and let $(S, \partial S)$ denote the submanifold pair, $\left(\theta(\partial_{0}M), \; \theta(\partial_{0,1}M)\right) \subset (\R^{\infty}_{+}, \partial\R^{\infty}_{+})$.
\begin{definition} \label{defn: moduli M} 
 We define
$\mathcal{M}(M)$ to be the set of compact $m$-dimensional submanifold triads 
$$
(M'; \partial_{0}M', \partial_{1}M') \;  \subset \; \left([0,\infty)\times\R^{\infty}_{+}; \; \{0\}\times\R^{\infty}_{+}, \; [0, \infty)\times\partial \R^{\infty}_{+}\right)
$$
such that:
\begin{enumerate} \itemsep.1cm
\item[(i)] $(\partial_{0}M', \partial_{0,1}M') = (S, \partial S)$;
\item[(ii)] $(M',\;  \partial_{1}M')$ contains $([0, \epsilon)\times S, \; [0, \varepsilon)\times\partial S)$ for some $\epsilon > 0$;
\item[(iii)] $(M'; \partial_{0}M', \partial_{1}M')$ is diffeomorphic to $(M; \partial_{0}M, \partial_{1}M)$.
\end{enumerate}
Denote by $\mathcal{E}(M)$ the space of embeddings, 
$(M; \partial_{0}M, \partial_{1}M)  \longrightarrow \left([0,\infty)\times\R^{\infty}_{+}; \; \{0\}\times\R^{\infty}_{+}, \; [0, \infty)\times\partial\R^{\infty}_{+}\right),$
topologized in the $C^{\infty}$-topology.
The space $\mathcal{M}(M)$ is topologized as a quotient of the space $\mathcal{E}(M)$ where two embeddings are identified if they have the same image. 
\end{definition}
It follows from Definition \ref{defn: moduli M} that $\mathcal{M}(M)$ is equal to the orbit space, $\mathcal{E}(M)/\Diff(M, \partial_{0}M)$.
By \cite[Lemma A.1]{P 15}, it follows that the quotient map $\mathcal{E}(M) \longrightarrow \mathcal{E}(M)/\Diff(M, \partial_{0}M) = \mathcal{M}(M)$
is a locally trivial fibre-bundle. 
In \cite{G 12} it is proven that the space $\mathcal{E}(M)$ is weakly contractible.
These two facts together imply that there is a weak-homotopy equivalence,
$\mathcal{M}(M) \simeq \BDiff(M, \partial_{0}M),$
and thus we may take the space $\mathcal{M}(M)$ to be a model for the classifying space of the diffeomorphism group $\Diff(M, \partial_{0}M)$.

Now let $p$ and $q$ be positive integers with $p + q + 1 = m$.
Recall from Section \ref{section: Preliminary Constructions} the relative cobordism, $(K_{p,q}, \partial_{1}K_{p,q}): (\partial_{0}M\times\{0\}, \partial_{0,1}M\times\{0\}) \rightsquigarrow (\partial_{0}M\times\{1\}, \partial_{0,1}M\times\{1\})$.
Choose an embedding, 
$$\alpha: (K_{p,q}; \; \partial_{0}K_{p,q}, \; \partial_{1}K_{p,q}) \; \longrightarrow \; ([0,1]\times\R^{\infty}_{+}; \; \{0,1\}\times\R^{\infty}_{+}, \; [0,1]\times\partial\R^{\infty}_{+}),$$ 
that satisfies $\alpha(i, x) = (i, \theta(x))$ for all $(i, x) \in \{0,1\}\times\partial_{0}M = \partial_{0}K_{p,q}$.
For a submanifold $M' \subset [0,\infty)\times\R^{\infty}_{+}$, denote by 
$M' + e_{1} \subset [1,\infty)\times\R^{\infty}_{+}$
the submanifold obtained by translating $M'$ over $1$-unit in the first coordinate. 
For $M' \in \mathcal{M}(M)$, the submanifold 
$\alpha(K_{p,q})\cup(M'\cup e_{1}) \; \subset \; [0,\infty)\times\R^{\infty}_{+}$
is an element of the space $\mathcal{M}(M\cup_{\partial_{0}}K_{p,q})$, and thus we have a continuous map,
\begin{equation} \label{ref: p-stabilization map}
s_{p,q}: \mathcal{M}(M) \longrightarrow \mathcal{M}(M\cup_{\partial_{0}}K_{p,q}); \quad M' \mapsto \alpha(K_{p,q})\cup(M' + e_{1}).
\end{equation}
The construction of $s_{p,q}$ depends on the choice of embedding $\alpha$. 
Any two such embeddings are isotopic and thus it follows that the homotopy class of $s_{p,q}$ does not depend on any such choice.  
Using the weak homotopy equivalence $\mathcal{M}(M) \simeq \BDiff(M, \partial_{0}M)$, it follows that the homotopy class $s_{p,q}$ agrees with the homotopy class of the stabilization map, $\BDiff(M, \partial_{0}M) \longrightarrow \BDiff(M\cup_{\partial_{0}}K_{p,q}, \partial_{0}M)$ defined in Section \ref{section: Preliminary Constructions}.

\subsection{A Semi-Simplicial Resolution}
Let $(M; \partial_{0}M, \partial_{1}M)$ be an $m$-dimensional manifold triad as in Section \ref{subsection: semi-simplicial resolution}. 
Fix $p, q \in \N$ such that $p+q+1 = m$. 
We now construct a semi-simplicial resolution of the moduli space $\mathcal{M}(M)$.
Choose once and for all a coordinate patch $c: (\R^{m-1}_{+}, \; \partial\R^{m-1}_{+}) \hookrightarrow (S, \; \partial S)$. 
 Such a coordinate patch induces for each $M' \in \mathcal{M}(M)$ a germ of an embedding 
 $[0, 1)\times\R^{m-1}_{+} \; \longrightarrow \; M'$ 
 as in the definition of $\bar{K}^{\partial}_{\bullet}(M')_{p,q}$ (see Definition \ref{defn: the embedding complex (semi simp)}). 
For each non-negative integer $l$, we define $X^{\partial}_{l}(M)_{p,q}$ to be the set of pairs $(M', \phi)$ where $M' \in \mathcal{M}(M)$ and $\phi \in \bar{K}^{\partial}_{l}(M')_{p,q}$.
The space $X^{\partial}_{l}(M)_{p,q}$ is topologized as the quotient, 
$$X^{\partial}_{l}(M)_{p,q} = (\mathcal{E}(M)\times \bar{K}^{\partial}_{l}(M)_{p,q})/\Diff(M, \partial_{0}M).$$
The assignments $[l] \mapsto X^{\partial}_{l}(M)_{p,q}$ make $X^{\partial}_{\bullet}(M)_{p,q}$ into a semi-simplicial space where the face maps are induced by the face maps in $\bar{K}^{\partial}_{\bullet}(M)_{p,q}$.
The forgetful maps 
$X^{\partial}_{l}(M)_{p,q}  \longrightarrow \mathcal{M}(M),$ $(M', \phi) \mapsto M',$
assemble to yield the augmented semi-simplicial space, 
$X^{\partial}_{\bullet}(M)_{p,q} \longrightarrow X^{\partial}_{-1}(M)_{p,q},$
where the space $X^{\partial}_{-1}(M)_{p,q}$ is set equal to $\mathcal{M}(M)$. 
We have the following proposition:
\begin{proposition} \label{prop: highly connected resolution}
The augmentation map $|X^{\partial}_{\bullet}(M)_{p,q}| \longrightarrow X^{\partial}_{-1}(M)_{p,q}$ is $\frac{1}{2}(r_{p,q}(M) - 2 - d)$-connected, where $d = d(\pi_{p}(S^{q}))$ is the generating set length.  
\end{proposition} 
\begin{proof}
For each $l \in \Z_{\geq 0}$ the forgetful map $X^{\partial}_{l}(M)_{p,q}  \longrightarrow \mathcal{M}(M)$ is a locally trivial fibre bundle with fibre given by the space $\bar{K}^{\partial}_{l}(M)_{p,q}$.
It follows that $|X^{\partial}_{\bullet}(M)_{p,q}| \longrightarrow X^{\partial}_{-1}(M)_{p,q}$ is a locally trivial fibre bundle as well with fibre given by $|\bar{K}^{\partial}_{\bullet}(M)_{p,q}|$.
The proposition then follows from Corollary \ref{corollary: high connectivity of main flag complex} using the long exact sequence on homotopy groups associated to a fibre-sequence. 
\end{proof}

\subsection{Proof of Theorem \ref{theorem: main homological stability theorem}} \label{section: homological stability final}
We now will show how to use the augmented semi-simplicial space  
$X^{\partial}_{\bullet}(M)_{p,q} \rightarrow X^{\partial}_{-1}(M)_{p,q}$
to complete the proof of Theorem \ref{theorem: main homological stability theorem}.
First, we fix some new notation which will make the steps of the proof easier to state.
For what follows let $(M; \partial_{0}M, \partial_{1}M)$ be a compact $m$-dimensional manifold triad with non-empty boundary. 
As in the previous sections, choose positive integers $p$ and $q$ with $p + q + 1 = m$ that satisfy the inequalities (\ref{equation: fundamental inequalities}) with respect to $(M, \partial_{1}M)$.
We work with the same choice of $p$ and $q$ for the rest of the section.
For each $g \in \N$ we denote by $M_{g}$ the manifold obtained by forming the boundary connected-sum of $M$ with $(D^{p+1}\times S^{q})^{\# g}$ along the face $\partial_{1}M$.  
Clearly we have $r_{p,q}(M_{g}) \geq g$. 
We consider the spaces $\mathcal{M}(M_{g})$.
For each $g \in \N$ we have the stabilization map 
$s_{p,q}: \mathcal{M}(M_{g}) \longrightarrow \mathcal{M}(M_{g+1})$.
Theorem \ref{theorem: main homological stability theorem} translates to the following statement:
\begin{theorem} \label{thm: main theorem neq notation}
The induced map
$(s_{p,q})_{*}: H_{k}(\mathcal{M}(M_{g}); \Z) \longrightarrow H_{k}(\mathcal{M}(M_{g+1}); \Z)$
is an isomorphism when $k \leq \frac{1}{2}(g - 3 - d)$ and an epimorphism when $k \leq \frac{1}{2}(g - 1 - d)$.
\end{theorem}

Since $r_{p,q}(M_{g}) \geq g$,
Proposition \ref{prop: highly connected resolution} implies that the map 
$|X^{\partial}_{\bullet}(M_{g})_{p,q}| \longrightarrow X^{\partial}_{-1}(M_{g})_{p,q}$ is $\frac{1}{2}(g - 2-d)$-connected. 
For each pair of integers $g, k \in \Z_{\geq 0}$ with $k < g$, there is a map
\begin{equation} \label{eq: resolution level map}
F^{g}_{k}: \mathcal{M}(M_{g-k-1}) \longrightarrow X^{\partial}_{k}(M_{g})_{p,q}
\end{equation}
defined as follows.
Let $K^{k}_{p,q} \subset [0, k+1]\times\R^{\infty}_{+}$ denote the $(k+1)$-fold concatenation of the submanifold $\alpha(K_{p,q}) \subset [0, 1]\times\R^{\infty}_{+}$ used in the construction of stabilization map (\ref{ref: p-stabilization map}), i.e.
$$
K^{k}_{p,q} \; = \; \alpha(K_{p,q})\cup[\alpha(K_{p,q}) + e_{1}]\cup\cdots\cup [\alpha(K_{p,q}) + k\cdot e_{1}].
$$
For each $k \in \Z_{\geq 0}$ we fix a $k$-simplex $(\zeta_{0}, \dots, \zeta_{k}) \in \bar{K}^{\partial}(K^{k}_{p,q})_{p,q}$.
The map $F^{g}_{k}$ from (\ref{eq: resolution level map}) is defined by the formula,
$
F^{g}_{k}(M') \; = \; \left((\zeta_{0}, \dots, \zeta_{k}), \; \; K^{k}_{p, q}\cup[(k+1)\cdot e_{1} + M']\right).
$
It follows directly from the definition of $F^{g}_{k}$ that for each pair $k < g$, the diagram 
\begin{equation} \label{equation: stabilization -face map diagram}
\xymatrix{
\mathcal{M}(M_{g-k-1}) \ar[d]^{F^{g}_{k}} \ar[rr]^{S_{p,q}} && \mathcal{M}(M_{g-k}) \ar[d]^{F^{g}_{k}} \\
X^{\partial}_{k}(M_{g})_{p,q} \ar[rr]^{d_{k}} && X^{\partial}_{k-1}(M_{g})_{p,q}
}
\end{equation}
is commutative. 
The following proposition is proven in the same way as \cite[Propositions 5.3 and 5.5]{GRW 12}.
For this reason we omit the proof and refer the reader to these analogous results from \cite{GRW 12} and \cite{GRW 14} for details. 
\begin{proposition} \label{prop: homotopy commutativity} Let $g \geq 4 + d$. We have the following:
\begin{enumerate} \itemsep.2cm
\item[(i)] For each $k < g$, the map $F^{g}_{k}: \mathcal{M}(M_{g-k-1}) \longrightarrow X^{\partial}_{k}(M_{g})_{p,q}$ is a weak homotopy equivalence. 
\item[(ii)] The face maps $d_{i}: X^{\partial}_{k}(M_{g})_{p,q} \longrightarrow X^{\partial}_{k-1}(M_{g})_{p,q}$ are weakly homotopic. 
\end{enumerate}
\end{proposition}

To finish the proof of Theorem \ref{thm: main theorem neq notation}, consider the spectral sequence associated to the augmented semi-simplicial space $X^{\partial}_{\bullet}(M_{g})_{p,q} \rightarrow X^{\partial}_{-1}(M_{g})_{p,q}$
with $E^{1}$-term given by 
$E^{1}_{j,l} = H_{j}(X^{\partial}_{l}(M_{g})_{p,q}; \; \Z)$ for $l \geq -1$ and $j \geq 0$.
The differential is given by $d^{1} = \sum(-1)^{i}(d_{i})_{*}$, where $(d_{i})_{*}$ is the map on homology induced by the $i$th face map in $X^{\partial}_{\bullet}(M_{g})_{p,q}$. 
The group $E^{\infty}_{j, l}$ is a subquotient of the relative homology group 
$$H_{j+l+1}(X^{\partial}_{-1}(M_{g})_{p,q}, |X^{\partial}_{\bullet}(M_{g})_{p,q}|; \; \Z).$$
Proposition \ref{prop: homotopy commutativity} together with Corollary \ref{prop: highly connected resolution} and commutativity of diagram (\ref{equation: stabilization -face map diagram}) imply the following facts:
\begin{enumerate} \itemsep.2cm
\item[(a)] For $g \geq 4 + d$, there are isomorphisms $E^{1}_{j,l} \cong H_{l}(\mathcal{M}(M_{g-j-1});\; \Z).$
\item[(b)] The differential
$d^{1}: H_{l}(\mathcal{M}(M_{g-j-1}); \; \Z) \cong E^{1}_{j,l} \longrightarrow E^{1}_{j-1,l} \cong H_{l}(\mathcal{M}(M_{g-j}); \; \Z)$ is equal to 
$(s_{p,q})_{*}$ when $j$ is even and is equal to zero when $j$ is odd.
\item[(c)] The term $E^{\infty}_{j,l}$ is equal to $0$ when $j +l \leq \frac{1}{2}(g-2-d)$. 
\end{enumerate}
To complete the proof one uses (c) to prove that the differential $d^{1}: E^{1}_{2j,l} \longrightarrow E^{1}_{2j-1,l}$ is an isomorphism 
when $0 < j \leq \frac{1}{2}(g - 3 - d)$  and an epimorphism when $0 < j \leq \frac{1}{2}(g -1 - d)$. 
This is done by carrying out the exact inductive argument given in \cite[Section 5.2: \textit{Proof of Theorem 1.2}]{GRW 12}. This establishes Theorem \ref{thm: main theorem neq notation} and the main result of this paper, Theorem \ref{theorem: main homological stability theorem}.

\section{Tangential Structures} \label{section: tangential structures}
In this section we prove an analogue of Theorem \ref{theorem: main theorem 1} for the moduli spaces of manifolds equipped with tangential structures. 
Recall that a tangential structure is a map $\theta: B \longrightarrow BO(d)$. 
A $\theta$-structure on a $d$-dimensional manifold $M$ is a bundle map (fibrewise linear isomorphism) $\ell: TM \longrightarrow \theta^{*}\gamma^{d}$. 
More generally, if $M$ is an $l$-dimensional manifold with $l \leq d$, then a $\theta$-structure on $M$ is a bundle map $TM\oplus\epsilon^{d-l} \longrightarrow \theta^{*}\gamma^{d}$. 

Fix a tangential structure $\theta: B \longrightarrow BO(d)$. 
Let $M$ be a $d$-dimensional manifold with boundary.
Let $P \subset \partial M$ be a codimension-$0$ submanifold and let $\ell_{P}: TP\oplus\epsilon^{1} \longrightarrow \theta^{*}\gamma^{d}$ be a $\theta$-structure. 
We define  
$$\Bun(TM, \theta^{*}\gamma^{d}; \ell_{P}) \subset \Bun(TM, \theta^{*}\gamma^{d})$$ 
to be the subspace consisting of those $\theta$-structures on $M$ that agree with $\ell_{P}$ when restricted to $P$. 
The formula, 
$
\Bun(TM, \theta^{*}\gamma^{d}; \ell_{P})\times\Diff(M, P) \; \longrightarrow \; \Bun(TM, \theta^{*}\gamma^{d}; \ell_{P}), \quad (\ell, f)$,  $\mapsto \; \ell\circ Df,
$
defines a continuous action of the topological group $\Diff(M, P)$ on the space $\Bun(TM, \theta^{*}\gamma^{d}; \ell_{P})$.
We define $\BDiff_{\theta}(M, \ell_{P})$ to be the \textit{homotopy quotient}, $\Bun(TM, \theta^{*}\gamma^{d}; \ell_{P})//\Diff(M, P)$.

We proceed to construct stabilization maps analogous to those defined in Section \ref{section: Preliminary Constructions}. 
This will require us to make some choices. 
Let $p, q \in \Z_{\geq 0}$ be integers such that $p + q + 1 = d$. 
\begin{definition} \label{defn: canonical theta structure}
Fix once and for all a bundle map $\tau: \R^{d} \longrightarrow \theta^{*}\gamma^{d}$. 
This choice determines a canonical $\theta$-structure on any framed $d$-dimensional manifold. 
If $X$ is any such framed $d$-dimensional manifold, we denote this canonical $\theta$-structure on $X$ by $\ell^{\tau}_{X}$. 
Choose once and for all a framing, $TV^{1}_{p,q} \cong V^{1}_{p,q}\times\R^{p+q+1},$
and consider the canonical $\theta$-structure, $\ell^{\tau}_{V^{1}_{p,q}}$, induced by this chosen framing. 
We call a $\theta$-structure $\ell: TV^{1}_{p,q} \longrightarrow \theta^{*}\gamma^{d}$ \textit{standard} if it is homotopic to the canonical $\theta$-structure $\ell^{\tau}_{V^{1}_{p,q}}$. 
\end{definition}

Let $(M; \partial_{0}M , \partial_{1}M)$ be a $d$-dimensional manifold triad with $\partial_{0}M$ and $\partial_{1}M$ non-empty. 
Let $(K_{p,q}; \partial_{1}K_{p,q})$ be the relative cobordism between 
$(\partial_{0}M\times\{0\}, \partial_{0,1}M\times\{0\})$ and $(\partial_{0}M\times\{1\}, \partial_{0,1}M\times\{1\})$, introduced in Section \ref{section: Preliminary Constructions}.
Let us denote $P := \partial_{0}M$ and fix a $\theta$-structure $\ell_{P}: TP\oplus\epsilon^{1} \longrightarrow \theta^{*}\gamma^{d}$. 
Choose a $\theta$-structure $\ell_{K_{p,q}}: TK_{p,q} \longrightarrow \theta^{*}\gamma^{d}$ that agrees with $\ell_{P}$ on both components of $\partial_{0}K_{p,q} = \partial_{0}M\times\{0,1\}$, and that is standard (in the sense of Definition \ref{defn: canonical theta structure}) when restricted to $V^{1}_{p,q}$ (where $V^{1}_{p,q}$ is considered as a submanifold of $K_{p,q} = (\partial_{0}M\times[0,1])\natural V^{1}_{p,q})$. 
With this choice of $\theta$-structure we obtain a map,
$$
\Bun(TM, \theta^{*}\gamma^{d}; \ell_{P}) \; \longrightarrow \; \Bun(T(M\cup_{P}K_{p,q}), \theta^{*}\gamma^{d}; \ell_{P}), \quad \ell \; \mapsto \; \ell\cup\ell_{K_{p,q}}.
$$
This map is $\Diff(M, P)$-equivariant, and thus it induces a map, 
\begin{equation} \label{equation: theta stabilization}
s^{\theta}_{p,q}: \textstyle{\BDiff_{\theta}}(M, \ell_{P}) \; \longrightarrow \; \textstyle{\BDiff_{\theta}}(M\cup K_{p,q}, \ell_{P}). 
\end{equation}
In addition to the inequalities imposed on $p$ and $q$ in the statement of Theorem \ref{theorem: main homological stability theorem}, the following theorem will require us to also impose the further condition, $q \leq p$, and to assume that $\theta: B \longrightarrow BO(d)$ is such that the space $B$ is $q$-connected.   
\begin{theorem} \label{theorem: main homological stability theorem theta structure}
Let $p$ and $q$ be positive integers with $p + q +1 = d = \dim(M)$ and suppose that the inequalities of (\ref{equation: fundamental inequalities}) are satisfied. 
Suppose further that $q \leq p$ and that $\theta: B \longrightarrow BO(d)$ is such that $B$ is $q$-connected. 
Suppose that $r_{p,q}(M) \geq g$. 
Then the homomorphism,
$$\xymatrix{
(s^{\theta}_{p, q})_{*}: H_{k}(\BDiff_{\theta}(M, \ell_{P}); \; \Z) \longrightarrow H_{k}(\BDiff_{\theta}(M\cup_{P}K_{p,q}, \; \ell_{P}); \; \Z),
}$$
is an isomorphism when $k \leq \tfrac{1}{2}(g - 4)$ and an epimorphism when $k \leq \tfrac{1}{2}(g - 2)$.
\end{theorem}
The proof of the above theorem is similar to Theorem \ref{theorem: main homological stability theorem} and requires only slight modifications. 
The main ingredient of the proof is to show that the tangentially structured analogue of the simplicial complex $K^{\partial}(M)_{p,q}$ (see Definition \ref{defn: the embedding complex theta} below) is highly-connected relative to the rank $r_{p,q}(M)$. 
With this high-connectivity established, the proof of Theorem \ref{theorem: main homological stability theorem theta structure} follows in exactly the same way as the proof of Theorem \ref{theorem: main homological stability theorem}, as outlined in Section \ref{section: homological stability}. 
We will show explicitly how to prove high-connectivity of the complex (Proposition \ref{theorem: highly connected complex theta}), and refer the reader to \cite[Section 7]{GRW 14} for the rest of the argument, which by this point is standard.

We proceed to construct a simplicial complex (and related semi-simplicial spaces) analogous to the one constructed in Section \ref{section: simplicial complexes}. 
Let $P$ and $\ell_{P}: TP\oplus\epsilon^{1} \longrightarrow \theta^{*}\gamma^{d}$ be as in the statement of Theorem \ref{theorem: main homological stability theorem theta structure}. 
Let $(M; \partial_{0}M, \partial_{1}M)$ be a $d$-dimensional compact manifold triad, with $\partial_{0}M = P$. 
Choose a $\theta$-structure $\ell_{M} \in \Bun(TM, \theta^{*}\gamma^{d}; \ell_{P})$.
The definition below should be compared to \cite[Definition 7.10]{GRW 14}.
\begin{definition} \label{defn: the embedding complex theta} 
Let
$a: [0,1)\times\R^{d-1}_{+} \longrightarrow M$ be an embedding with
$a^{-1}(\partial_{0}M) =  \{0\}\times\R^{d-1}_{+}$ and $a^{-1}(\partial_{1}M) =  [0, 1)\times\partial\R^{d-1}_{+}.$
For each pair of positive integers $p$ and $q$ with $p+q+1 = d$, 
we define a simplicial complex
  $K^{\partial}(M, \ell_{M}, a)_{p,q}$ as follows:
\begin{enumerate}  \itemsep.2cm
\item[(i)] 
A vertex in $K^{\partial}(M, \ell_{M}, a)_{p,q}$ is defined to be a triple $(t, \phi, \gamma)$, where
$(t, \phi)$ is an element of $K^{\partial}(M, a)_{p,q}$, and $\gamma$ is a path in $\Bun(\widehat{V}_{p,q}, \theta^{*}\gamma^{d})$, starting at $\phi^{*}\ell_{M}$, ending at $\ell_{\widehat{V}_{p,q}}$, and constant on the subset $D^{p+q}_{+}\times\{0\} \subset \widehat{V}_{p,q}$. 
\item[(ii)] A set of vertices $\{(t_{0}, \phi_{0}, \gamma_{0}), \dots, (t_{l}, \phi_{l}, \gamma_{l})\}$ forms an $l$-simplex 
if 
$t_{i} \neq t_{j}$ and 
$\phi_{i}(B_{p,q}) \cap \phi_{j}(B_{p,q}) = \emptyset$ whenever $i \neq j$, just as in Definition \ref{defn: the embedding complex}.  
\end{enumerate}
 As in Section \ref{section: simplicial complexes} we will also need to work with a semi-simplicial space $K^{\partial}_{\bullet}(M, \ell_{M}, a)_{p,q}$ analogous to the simplicial complex defined above.
\begin{enumerate}  \itemsep.2cm
\item[(i)] 
The space of $0$-simplices $K^{\partial}_{0}(M, \ell_{M}, a)_{p,q}$ is defined to have the same underlying set as the set of vertices of the simplicial complex $K^{\partial}(M, \ell_{M}, a)_{p,q}$. 
\item[(ii)] The space of $l$-simplices, $K^{\partial}_{l}(M, \ell_{M}, a) \subset (K^{\partial}_{0}(M, \ell_{M}, a))^{l+1}$ consists of the ordered $(l+1)$-tuples, 
$$((t_{0}, \phi_{0}, \gamma_{0}), \dots, (t_{l}, \phi_{l}, \gamma_{l})),$$
such that $t_{0} < \cdots < t_{l}$ and $\phi_{i}(B_{p,q})\cap\phi_{j}(B_{p,q}) = \emptyset$ when $i \neq j$.
\end{enumerate}
The spaces $K^{\partial}_{l}(M,  \ell_{M}, a)\subset (\R\times
  \Emb(\widehat{V}_{p,q}, M)\times\Bun(T\widehat{V}_{p,q}, \theta^{*}\gamma^{d}))^{l+1}$ are topologized using the $C^{\infty}$-topology on
the spaces of embeddings and bundle maps.
The assignments $[l] \mapsto K^{\partial}_{l}(M, \ell_{M}, a)$ define a semi-simplicial space denoted by 
$K^{\partial}_{\bullet}(M, \ell_{M}, a)_{p,q}$ with face maps defined the same way as in Definition \ref{defn: the embedding complex (semi simp)}. 

 Finally, the sub-semi-simplicial space
$\bar{K}^{\partial}_{\bullet}(M, \ell_{M}, a)_{p,q} \subset K^{\partial}_{\bullet}(M, \ell_{M}, a)_{p,q}$ is defined to be the sub-semi-simplicial space consisting of all simplices $((t_{0}, \phi_{0}, \gamma_{0}), \dots, (t_{l}, \phi_{l}, \gamma_{l})) \in K^{\partial}_{l}(M, \ell_{M}, a)$ such that $\phi_{i}(\widehat{V}_{p,q}) \cap \phi_{j}(\widehat{V}_{p,q}) = \emptyset$ whenever $i \neq j$.
\end{definition} 

The key technical result that we will need is the lemma stated below. 
This lemma is the source of the requirement that $q \leq p$ and that the space $B$ be $q$-connected.
Fix a $\theta$-structure $\ell_{D}$ on the disk $D^{d-1}$ and fix an embedding $D^{d-1} \hookrightarrow \partial V_{p,q}$.
We consider the space $\Bun(TV_{p,q}, \theta^{*}\gamma^{d}; \ell_{D})$.
\begin{lemma} \label{lemma: transitivity of diffeomorphisms}
Suppose that $q \leq p$ and that $\theta: B \longrightarrow BO(d)$ is chosen so that $B$ is $q$-connected. 
Then given any two elements $\ell_{1}, \ell_{2} \in \Bun(TV_{p,q}, \theta^{*}\gamma^{d}; \ell_{D})$, there exists a diffeomorphism $f \in \Diff(V_{p,q}, D^{p+q})$ such that $\ell_{1}$ is homotopic to $\ell_{2}\circ Df$.
\end{lemma}
\begin{proof}
Since the space $B$ is $q$-connected and $V_{p,q}$ is homotopy equivalent to $S^{q}$, it follows that the underlying map $V_{p,q} \longrightarrow B$ of any $\theta$-structure on $V_{p,q}$ is null-homotopic. 
It follows that every $\theta$-structure on $V_{p,q}$ is homotopic to one that is induced by a framing of the tangent bundle. 
That is, every $\ell \in \Bun(TV_{p,q}, \theta^{*}\gamma^{d}; \ell_{D})$ is homotopic to a $\theta$-structure of the form,
$$
\xymatrix{
TV_{p,q} \ar[r]^{\cong \ \ \ \ \ } & V_{p,q}\times\R^{d} \ar[r]^{ \ \ \ \ \text{pr}} & \R^{d} \ar[r]^{\tau} & \theta^{*}\gamma^{d},
}
$$
where the first arrow is a framing of the tangent bundle. 
Fix a framing, 
$\phi_{D}: TD^{d-1}\oplus\epsilon^{1} \stackrel{\cong} \longrightarrow D^{d-1}\times\R^{d}.$ 
Let $\Fr(TV_{p,q}, \phi_{D})$ denote the space of framings of $TV_{p,q}$ that agree with $\ell_{D}$ when restricted to the disk $D^{d-1} \subset \partial V_{p, q}$. 
From the observation made above, to prove the lemma it will suffice to prove the following statement: given any two framings $\varphi_{1}, \varphi_{2} \in \Fr(TV_{p,q}, \phi_{D})$, there exists $f \in \Diff(V_{p,q}, D^{d})$ such that $\varphi_{1}\circ Df$ is homotopic (through framings) to $\varphi_{2}$.
Let $\varphi_{1}, \varphi_{2} \in \Fr(TV_{p,q}, \phi_{D})$.
The tangent bundle $TV_{p,q}$ has a natural splitting $E\oplus N  \cong TV_{p,q}$. 
The bundle $E$ is the pull-back of the tangent bundle $TS^{q}$ over the projection $V_{p,q} \longrightarrow S^{q}$ and $N$ is the pull-back of the normal bundle of $S^{q} \hookrightarrow V_{p,q}$ over the projection $V_{p,q} \longrightarrow S^{q}$.
The bundle $N \longrightarrow V_{p,q}$ is trivial and has fibres of dimension $p+1$. 
Pick once and for all a standard framing, 
$$\phi: TV_{p,q} \stackrel{\cong} \longrightarrow V_{p,q}\times\R^{d}.$$
Since $p \geq q$ by assumption, the stabilization map $\pi_{q}(SO_{p+1}) \longrightarrow \pi_{q}(SO)$ is surjective. 
From this is follows that $\varphi_{i}$ (for $i = 1, 2$) is homotopic to a framing of the form 
\begin{equation} \label{equation: factored framing}
\xymatrix{
TV_{p,q} \ar[r]^{\cong} & E\oplus N \ar[rr]^{\Id_{E}\oplus\widehat{\varphi}_{i}} && E\oplus N \ar[r]^{\cong} & TV_{p,q}  \ar[r]^{\phi \ \ \ \ \ } & V_{p,q}\times\R^{d},
}
\end{equation}
for some bundle isomorphism $\widehat{\varphi}_{i}: N \stackrel{\cong} \longrightarrow N$.  
To prove the proposition it will suffice to show the following:
given two bundle isomorphisms $\alpha_{1}, \alpha_{2}: N \stackrel{\cong} \longrightarrow N$ that agree when restricted to $D^{d-1} \subset \partial V_{p,q}$, there exists $f \in \Diff(V_{p,q}, D^{d-1})$ such that $\alpha_{1}\circ Df = \alpha_{2}$. 
Let $D(N)$ denote the \textit{disk bundle} associated to $N$. 
Let $\bar{\alpha}_{1}, \bar{\alpha}_{2}: D(N) \stackrel{\cong} \longrightarrow D(N)$ be the bundle maps induced by $\alpha_{1}$ and $\alpha_{2}$. 
Choose a diffeomorphism $h: V_{p,q} \stackrel{\cong} \longrightarrow  D(N)$. 
By setting $f$ equal to the diffeomorphism given by, 
$\xymatrix{
V^{1}_{p,q} \ar[r]^{h} & D(N) \ar[rr]^{\bar{\alpha}_{2}\circ\bar{\alpha}_{1}^{-1}} && D(N)  \ar[r]^{h^{-1}} & V^{1}_{p,q},}$
it follows that $\bar{\alpha}_{2}\circ D(f) = \bar{\alpha}_{1}$. 
This concludes the proof of the lemma.
\end{proof}

The main technical ingredient in the proof of Theorem \ref{theorem: main homological stability theorem theta structure} is the following proposition. 
\begin{proposition} \label{theorem: highly connected complex theta}
Let $(M; \partial_{0}M, \partial_{1}M)$ and $p + q + 1 = d$ satisfy the inequalities (\ref{equation: required theorem inequalities}). 
Suppose further that $q \leq p$ and that $\theta: B \longrightarrow BO(d)$ is chosen so that $B$ is $q$-connected. 
Let $r_{p,q}(M) \geq g$. 
Then for any $\ell_{M} \in \Bun(TM, \theta^{*}\gamma^{d}; \ell_{P})$, the geometric realization $|\bar{K}^{\partial}_{\bullet}(M, \ell_{M}, a)_{p,q}|$ is $\tfrac{1}{2}(g - 4)$-connected. 
\end{proposition}
\begin{proof}
As before, the degree of connectivity of $|\bar{K}^{\partial}_{\bullet}(M, \ell_{M}, a)_{p,q}|$ is bounded below by the degree of connectivity of $|K^{\partial}(M, \ell_{M}, a)_{p,q}|$. 
To prove the theorem it will suffice to show that $|K^{\partial}(M, \ell_{M}, a)_{p,q}|$ is $\tfrac{1}{2}(g - 4)$-connected.
Let $\mathcal{T} \subset K^{\partial}(M, \ell_{M}, a)_{p,q}\times K^{\partial}(M, \ell_{M}, a)_{p,q}$ be the symmetric relation from Definition \ref{defn: transversality relation}. 
As in Section \ref{section: high connectivity of K} we have a simplicial map 
$F_{p,q}: K^{\partial}(M, \ell_{M}, a)_{p,q} \; \longrightarrow \; L(\mathcal{W}^{\partial}_{p,q}(M, \partial_{1}M))$
defined by sending $(t, \phi, \gamma)$ to the morphism of Wall forms $\mb{W}_{p,q} \longrightarrow \mathcal{W}^{\partial}_{p,q}(M, \partial_{1}M)$ induced by $\phi$. 
We will need to show that the map $F_{p,q}$ has the following properties: 
\begin{enumerate} \itemsep2pt
\item[(i)] the map $F_{p,q}$ has the link lifting property with respect to $\mathcal{T}$ (see Definition \ref{defn: cone lifting property}),
\item[(ii)] $F_{p,q}(\lk_{K^{\partial}(M, \ell_{M}, a)_{p,q}}(\zeta)) \leq \lk_{L(\mathcal{W}^{\partial}_{p,q}(M, \partial_{1}M))}(F_{p,q}(\zeta))$ for any simplex $\zeta \in K^{\partial}(M, \ell_{M}, a)_{p,q}$.
\end{enumerate}
Condition (ii) is proven in the same way as in the proof of Theorem \ref{theorem: high-connectivity}. 
The proof of condition (i) is similar to the proof of Theorem \ref{theorem: high-connectivity} but requires one extra step which we describe below. 
This step will rely on Lemma \ref{lemma: transitivity of diffeomorphisms} and thus requires the conditions that $q \leq p$ and $B$ is $q$-connected. 
Let $f: \mb{W} \longrightarrow \mathcal{W}^{\partial}_{p,q}(M, \partial_{1}M)$ be a morphism of Wall forms, i.e.\ a vertex of $L(\mathcal{W}^{\partial}_{p,q}(M, \partial_{1}M))$. 
Let 
$(\phi_{1}, \gamma_{1}, ), \dots, (\phi_{k}, \gamma_{k}) \; \in \; K^{\partial}(M, \ell_{M}, a)_{p,q}$
be a collection of vertices that is in general position with respect to $\mathcal{T}$, and such that 
$$F_{p,q}(\gamma_{i}, \phi_{i}) \in \lk(f) \quad \text{for $i = 1, \dots, k$,}$$
(we have dropped the numbers $t_{1}, \dots, t_{k}$ from the notation to save space). 
Let $(\phi'_{1}, \gamma'_{1}), \dots, (\phi'_{m}, \gamma'_{m}) \in K^{\partial}(M, \ell_{M}, a)_{p,q}$ be another arbitrary collection of vertices. 
To show that $F_{p,q}$ has the link lifting property with respect to $\mathcal{T}$, 
we need to construct a vertex $(\phi, \gamma) \in K^{\partial}(M, \ell_{M}, a)_{p,q}$ with $F_{p, q}(\phi, \gamma) = f$, such that 
$$(\phi_{i}, \gamma_{i}) \in \lk(\phi, \gamma) \quad \text{and} \quad ((\psi_{j}, \gamma_{j}), (\phi, \gamma)) \in \mathcal{T}$$
for all $i = 1, \dots, k$ and $j = 1, \dots, m$.
By using the same procedure employed in the proof of Theorem \ref{theorem: high-connectivity}, we may construct the embedding $\phi: \widehat{V}_{p,q} \longrightarrow M$ in the same way that was done there.   
However, in order to obtain the path $\gamma: [0, 1] \longrightarrow \Bun(T\widehat{V}_{p,q}, \theta^{*}\gamma^{d})$, we need to use Lemma \ref{lemma: transitivity of diffeomorphisms}. 
Since $q \leq p$ and the space $B$ is $q$-connected, by Lemma \ref{lemma: transitivity of diffeomorphisms} we may find a diffeomorphism $\varphi: \widehat{V}_{p,q} \longrightarrow \widehat{V}_{p,q}$, that is the identity on the half-disk $D^{p+q}_{+} \subset \partial\widehat{V}_{p,q}$, such that the $\theta$-structure $\varphi^{*}\phi^{*}\ell_{M}$ given by
\begin{equation} \label{equation: composition theta phi}
\xymatrix{
T\widehat{V}_{p,q} \ar[r]^{D\varphi} & T\widehat{V}_{p,q} \ar[r]^{D\phi} & TM \ar[r]^{\ell_{M}} & \theta^{*}\gamma^{d},
}
\end{equation}
is on the same path component of $\Bun(T\widehat{V}_{p,q}, \theta^{*}\gamma^{d})$ as the canonical $\theta$-structure $\ell^{\tau}_{\widehat{V}_{p,q}}$. 
Letting $\gamma: [0, 1] \longrightarrow \Bun(T\widehat{V}_{p,q}, \theta^{*}\gamma^{d})$ be a path from $\varphi^{*}\phi^{*}\ell_{M}$ to $\ell^{\tau}_{\widehat{V}_{p,q}}$ it follows that the pair $(\phi\circ\varphi, \gamma)$ is a vertex in the complex $K^{\partial}(M, \ell_{M}, a)_{p,q}$ that satisfies all of the desired conditions. 

The construction of this vertex concludes our verification of the link lifting property.  
It follows from Lemma \ref{lemma: link lift lemma} that the degree of connectivity of $|K^{\partial}(M, \ell_{M}, a)_{p,q}|$ is bounded below by the degree of connectivity of $|L(\mathcal{W}^{\partial}_{p,q}(M, \partial_{1}M))|$.
Since $q \leq p$, $\pi_{q}(S^{p})$ is either isomorphic to $\Z$ or is zero, thus the generating set length $d(\pi_{q}(S^{p}))$ is either equal to $1$ or zero. 
It follows from Theorem \ref{thm: high connectivity} that $|L(\mathcal{W}^{\partial}_{p,q}(M, \partial_{1}M))|$ is at least $\tfrac{1}{2}(g - 4)$-connected. 
By what was proven above it follows that $|\bar{K}^{\partial}_{\bullet}(M, \ell_{M}, a)_{p,q}|$ is $\tfrac{1}{2}(g - 4)$-connected as well. 
This concludes the proof of the proposition.
\end{proof}
With the above proposition established, the proof of Theorem \ref{theorem: main homological stability theorem theta structure} is obtained by implementing the same constructions from Section \ref{section: homological stability}. 
We omit the rest of the proof and refer the reader to \cite[Section 7]{GRW 14} for details.

\appendix
\section{Embeddings and Disjunction} \label{appendix: embeddings and disjunction}
In this section we prove a disjunction result for embeddings of manifolds with boundary.   
This result implies Theorem \ref{theorem: relative higher Whitney trick} which is one of the main technical ingredients used to prove that the complex $K(M)_{p,q}$ is highly connected. 
\begin{theorem} \label{theorem: boundary disjunction}
Let $(M, \partial M)$ be a manifold pair of dimension $m$. 
Let $(P, \partial P)$ and $(Q, \partial Q)$ be manifold pairs of dimensions $p$ and $q$ respectively, with $\partial P \neq \emptyset \neq \partial Q$.  
Let 
$f: (P, \partial P) \longrightarrow (M, \partial M)$ and  $g: (Q, \partial Q) \longrightarrow (M, \partial M)$
be smooth embeddings and suppose that the following conditions are met:
\begin{enumerate} \itemsep.1cm
\item[(i)] $m > p + q/2 +1$,  $m > q + p/2 + 1$;
\item[(ii)] $(P, \partial P)$, $(Q, \partial Q)$, $P$, and $Q$ are $(p+q-m)$-connected;
\item[(iii)] $(M, \partial M)$ and $M$ are $(p+q-m+1)$-connected.
\end{enumerate}
Then there exists an isotopy
$\psi_{s}: (P, \partial P) \longrightarrow (M, \partial M)$, $s \in [0, 1],$
such that: 
$\psi_{0} = f$
and 
$\psi_{1}(P)\cap g(Q) = \emptyset$. 
\end{theorem}
The proof of the above theorem is based on a technique developed by Hatcher
and Quinn from \cite{HQ 74}.  
We recall the results of Hatcher and Quinn in the following section, develop some new techniques in the sections that follow, and
then finish the proof of Theorem \ref{theorem: boundary
  disjunction} in Section \ref{subsection: proof of main theorem}.
\subsection{The Hatcher-Quinn invariant}
We now review the construction of Hatcher and Quinn from \cite{HQ 74}.
This construction involves the framed bordism groups of a space, twisted by a stable vector bundle. 
\begin{definition} \label{defn: normal bordism}
Let $X$ be a space and let \textcolor{black}{$\zeta$ be a stable vector}
bundle over $X$.  For an integer $n$, $\Omega^{\fr}_{n}(X;
\zeta)$ is defined to be the set of bordism classes of triples $(M, f,
F)$, where $M$ is a closed $n$-dimensional smooth manifold, $f: M
\longrightarrow X$ is a map, and $F: \nu_{M} \longrightarrow
f^{*}(\zeta)$ is an isomorphism of stable vector bundles covering
the identity map on $M$.
\end{definition}
Let $M$, $P$, and $Q$ be smooth manifolds of
dimensions $m$, $p$, and $q$ respectively.  Let $t$ denote the integer
$p + q - m$.  Let
$f: (P, \partial P) \longrightarrow (M, \partial M)$ and $g: (Q, \partial Q) \longrightarrow (M, \partial M)$
be smooth maps. 
We denote by $E(f, g)$ the \textit{homotopy pull-back} of the maps $f$ and $g$. 
Explicitly, $E(f, g)$ is the space defined by,
$$E(f, g) = \{(x, y, \gamma) \in P\times Q\times \text{Path}(M) \; | \; f(x) = \gamma(0), \quad g(y) = \gamma(1) \; \}.$$
Consider the diagram
\begin{equation}
\xymatrix{
E(f, g) \ar[rr]^{\pi_{P}} \ar[d]^{\pi_{Q}} \ar[drr]^{\hat{s}} && P \ar[d]^{f}   \\
Q \ar[rr]^{g} && M 
}
\end{equation}
where $\pi_{P}$ and $\pi_{Q}$ are the projection maps and $\hat{s}$ is
the map defined by $\hat{s}(x, y, \gamma) = \gamma(1/2)$.
Let $\nu_{P}$, $\nu_{Q}$ denote the stable normal
bundles associated to the manifolds $P$ and $Q$ respectively.  We
denote by $\eta(f, g)$ the stable vector bundle over $E(f, g)$ given
by the Whitney-sum
$\pi_{P}^{*}\nu_{P}\oplus\pi_{Q}^{*}\nu_{Q}\oplus\hat{s}^{*}(TM)$.  We
will need to consider the bordism group $\Omega^{\fr}_{t}(E(f, g);
\; \eta(f, g)).$

Suppose now that the maps $f: (P, \partial P) \longrightarrow (M, \partial M)$ and $g: (Q, \partial Q) \longrightarrow (M, \partial M)$ satisfy
$
f(\partial P)\cap g(\partial Q) = \emptyset.
$
It follows that the pull-back, 
$f\pitchfork g := (f\times g)^{-1}(\triangle_{M}) \subset P\times Q,$ 
is a closed submanifold of dimension $p+q-m$. 
Let us denote $t := p+q-m$.
Let 
$\iota: f\pitchfork g \longrightarrow E(f, g)$ 
denote the canonical embedding given by the formula $(x, y) \mapsto (x, y, c_{f(x)})$, where $c_{f(x)} \in \text{Path}(M)$ is the constant path at the point $f(x) \in M$. 
The lemma below follows from \cite[Proposition 2.1]{HQ 74}. 
\begin{lemma} \label{lemma: hatcher quinn invariant}
Let 
$f: (P, \partial P) \longrightarrow (M, \partial M)$ and $g: (Q, \partial Q) \longrightarrow (M, \partial M)$
be transversal smooth maps such that $f(\partial P)\cap g(\partial Q) = \emptyset$. 
Then there is a natural bundle isomorphism 
$\hat{\iota}: \nu_{f\pitchfork g} \stackrel{\cong} \longrightarrow \iota^{*}(\eta(f, g))$
so that the triple $(f\pitchfork g, \iota, \hat{\iota})$ determines a well-defined element of the bordism group $\Omega^{\fr}_{t}(E(f, g); \; \eta(f, g))$.
\end{lemma}
\begin{definition}
For transversal maps $f$ and $g$ with $f(\partial P)\cap g(\partial Q) = \emptyset$ as in the previous lemma, we will denote by 
$\alpha_{t}(f, g, M) \in \Omega^{\fr}_{t}(E(f, g); \;  \eta(f, g))$
the element determined by the the triple $(f\pitchfork g, \iota, \hat{\iota})$ given in Lemma \ref{lemma: hatcher quinn invariant}. 
\end{definition}

The main result from \cite{HQ 74} is the following theorem.
\begin{theorem} \label{theorem: hatcher quinn main theorem}
Let $f: (P, \partial P) \longrightarrow (M, \partial M)$ and $g: (Q, \partial Q) \longrightarrow (M, \partial M)$ be smooth embeddings such that $f(\partial P)\cap g(\partial Q) = \emptyset$.
Suppose further that 
$m > p + q/2 + 1$ and $m > p/2 + q + 1.$ 
If the class $\alpha_{t}(f, g, M)$ is equal to the zero element in $\Omega^{\fr}_{t}(E(f, g); \eta(f, g))$, then there exists an isotopy,
$\psi_{s}: (P, \partial P) \longrightarrow (M, \partial M),$
such that: $\psi_{0} = f$, $\psi_{s}|_{\partial P} = f$ for all $s \in [0, 1]$,
and $\psi_{1}(P)\cap g(Q) = \emptyset$.
\end{theorem}
The bordism group $\Omega^{\fr}_{t}(E(f, g); \eta(f, g))$ in general
can be quite difficult to compute.  However, in the case where $P$,
$Q$, and $M$ are all highly connected, the group reduces to a far
simpler object.
\begin{proposition} \label{proposition: framed bordism reduction}
Let $f: (P, \partial P) \longrightarrow (M, \partial M)$ and $g: (Q, \partial Q) \longrightarrow (M, \partial M)$ be smooth embeddings such that $f(\partial P)\cap g(\partial Q) = \emptyset$.
Suppose that $P$ and $Q$ are $(p+q-m)$-connected and that $M$ is $(p+q-m+1)$-connected.
Then the natural map $\Omega^{\fr}_{t}(\text{pt.}) \rightarrow \Omega^{\fr}_{t}(E(f, g), \eta(f, g))$ is an isomorphism. 
\end{proposition}
\begin{proof}
Let $P$ and $Q$ be $(p+q-m)$-connected and let $M$ be $(p+q-m+1)$-connected.
There is a fibre sequence $\Omega M \longrightarrow E(f, g) \longrightarrow P\times Q$. 
The long exact sequence on homotopy groups implies that the space $E(f, g)$ is $(p+q-m)$-connected.
The proof of the proposition follows from this. 
\end{proof}
Suppose that $P$ and $Q$ are $(p+q-m)$-connected and that $M$ is $(p+q-m+1)$-connected.
If $f: (P, \partial P) \longrightarrow (M, \partial M)$ and $g: (Q, \partial Q) \longrightarrow (M, \partial M)$ are smooth embeddings we may consider $\alpha_{t}(f, g; M)$ to be an element of the framed bordism group $\Omega^{\fr}_{t}(\text{pt.})$, where $t = p + q - m$.
\begin{remark}
The element $\alpha_{t}(f, g; M)$ is the obstruction to finding an
isotopy, relative to the boundary $\partial P$, that pushes $f(P)$ off
of $g(Q)$.  This element $\alpha_{t}(f, g; M)$ very well may be
non-zero for arbitrary $f$, $g$, and $M$, and thus it may appear that
Theorem \ref{theorem: boundary disjunction} is false.  However,
Theorem \ref{theorem: boundary disjunction} does not assert the
existence of an isotopy that fixes the boundary of $P$.  
Indeed, as will
be seen in the following sections, $\alpha_{t}(f, g; M)$ is not an
obstruction to the existence of an isotopy that is non-constant on the
boundary of $P$.
\end{remark}
\subsection{Relative Hatcher-Quinn Invariant}
We will have to consider relative
framed bordism groups.
\begin{definition}
Let $(X, A)$ be a pair of spaces and let $\zeta$ be a stable vector
bundle over $X$.  For an integer $n$, $\Omega^{\fr}_{n}((X, A),
\zeta)$ is defined to be the set of bordism classes of triples $(M, f,
F)$ where $(M, \partial M)$ is an $n$-dimensional manifold pair, $f:
(M, \partial M) \longrightarrow (X, A)$ is a map, and $F: \nu_{M}
\longrightarrow f^{*}(\zeta)$ is an equivalence class of bundle
isomorphisms as before.
\end{definition}
For any space pair $(X, A)$ and stable vector bundle $\zeta$ over $X$,
there is a long exact sequence of bordism groups,
$
\xymatrix{
\cdots \ar[r] & \Omega^{\fr}_{n}(A; \zeta|_{A}) \ar[r] &  \Omega^{\fr}_{n}(X; \zeta) \ar[r] & \Omega^{\fr}_{n}((X, A); \zeta) \ar[r] & \Omega^{\fr}_{n-1}(A; \zeta|_{A}) \ar[r] & \cdots
}
$
Using these relative bordism groups, we define a relative version of
the Hatcher-Quinn invariant.  Let
$f: (P, \partial P) \longrightarrow (M, \partial M)$ and $g: (Q, \partial Q) \longrightarrow (M, \partial M)$
be embeddings.  
Unlike the case in the previous section, we now
include the possibility that the intersection $f(\partial P)\cap
g(\partial Q)$ be non-empty.  If $f$ and $g$ are transversal (and by
this we mean that both $f$ and $g$ and $f|_{\partial P}$ and
$g|_{\partial Q}$ are transversal in the ordinary sense), then the pull-back $f\pitchfork g$
is a manifold with boundary given by
$\partial(f\pitchfork g) = f|_{\partial P}\pitchfork g|_{\partial Q} \subset \partial P\times\partial Q.$
We will need to construct a relative version of the bordism invariant that was defined in the previous section. 

Let $\partial E(f, g)$ denote the homotopy pull-back $E(f|_{\partial P}, g|_{\partial Q})$. 
The space $\partial E(f, g)$ embeds naturally as a subspace of $E(f, g)$. 
We have a map of pairs 
$$\iota: (f\pitchfork g, \; \partial(f\pitchfork g)) \longrightarrow (E(f, g), \; \partial E(f, g)), \quad (x, y) \mapsto (x, y, c_{f(x)}).$$
The restriction of $\eta(f, g)$ to $\partial E(f, g)$ is equal to the bundle $\eta(f|_{\partial P}, g|_{\partial Q})$. 
To save space we will let $\widehat{E}(f, g)$ denote the pair $(E(f, g), \partial E(f, g))$. 
We will need to consider the relative bordism group 
$\Omega^{\fr}_{t}(\widehat{E}(f, g), \eta(f, g)).$
Let, 
\begin{equation}
\widehat{\partial}: \Omega^{\fr}_{t}\left(\widehat{E}(f, g), \; \eta(f, g)\right) \longrightarrow \Omega^{\fr}_{t-1}\left(\partial E(f, g), \; \eta(f, g)|_{\partial E(f, g)}\right),
\end{equation}
be the boundary homomorphism in the long exact sequence associated to the pair $\widehat{E}(f, g)$.
Using the same construction from Lemma \ref{lemma: hatcher quinn invariant}, we obtain:
\begin{lemma} \label{lemma: relative hatcher-quinn}
Let $f: (P, \partial P) \longrightarrow (M, \partial M)$ and $g: (Q, \partial Q) \longrightarrow (M, \partial M)$ be transversal maps. 
Then the pullback manifold $f\pitchfork g$ determines a class 
$\alpha^{\partial}_{t}(f, g, M) \in \Omega^{\fr}_{t}(\widehat{E}(f, g), \eta(f, g)).$
 Furthermore, we have
 $$\widehat{\partial}(\alpha^{\partial}_{t}(f, g, M)) \; = \; \alpha_{t-1}(f|_{\partial P}, g|_{\partial Q}, \partial M).$$
 \end{lemma}
 Lemma \ref{lemma: relative hatcher-quinn} will be useful to us in order to prove the following result. 
 \begin{proposition} \label{proposition: boundary disjunction}
 Let $f: (P, \partial P) \longrightarrow (M, \partial M)$ and $g: (Q, \partial Q) \longrightarrow (M, \partial M)$
 be embeddings.
 Suppose that $(P, \partial P)$ and $(Q, \partial Q)$ are $(p+q-m)$-connected and that $(M, \partial M)$ is $(p+q-m+1)$-connected.
 Then there exists an isotopy 
 $\Psi_{s}: P \longrightarrow M$ with $s \in [0, 1]$, such that
 $\Psi_{0} = f$ and 
$\Psi_{1}(\partial P)\cap g(\partial Q) = \emptyset$. 
 \end{proposition}
 \begin{proof}
Let $t$ denote the integer $p + q - m$. 
The connectivity conditions in the statement of the proposition implies that the pair $(E(f, g), \partial E(f, g))$ is $t$-connected, and thus, the bordism group $\Omega^{\fr}_{t}(\widehat{E}(f, g), \eta(f, g))$ is trivial. 
It follows from this that 
$\widehat{\partial}(\widehat{\alpha}_{t}(f, g, M)) =
\alpha_{t-1}(f|_{\partial P}, g|_{\partial Q}, \partial M) = 0.$ 
We
may then apply Theorem \ref{theorem: hatcher quinn main theorem} to
obtain an isotopy $\psi_{s}: \partial P \longrightarrow \partial M$
with $\psi_{0} = f|_{\partial P}$, such that $\psi_{1}(\partial P)\cap
g(\partial Q) = \emptyset$.  The proof of the proposition then follows
by application of the \textit{isotopy extension theorem}.
\end{proof}
\subsection{Creating intersections}
In this section we develop a technique for creating intersections with
prescribed Hatcher-Quinn obstructions.  Let $M$ and $Q$ be oriented,
connected manifolds of dimension $m$ and $q$ respectively, and let
$g: (Q, \partial Q) \longrightarrow (M, \partial M)$
be an embedding. 
Let $r = m - q$ and let $f: (D^{r}, \partial D^{r}) \longrightarrow (M, \partial M)$ be a smooth embedding transverse to $g$ such that, 
$f(\partial D^{r})\cap g(\partial Q) = \emptyset.$
Let $j \geq 0$ be an integer strictly less than $r$. 
With $j$-chosen in this way it follows that $\pi_{r+j}(S^{r})$ is in the stable range and thus we have an isomorphism, $\pi_{r+j}(D^{r}, \partial D^{r}) \cong \pi_{r+j}(S^{r})$.
Let $\varphi:
(D^{r+j}, \partial D^{r + j}) \longrightarrow (D^{r}, \partial D^{r})$
be a smooth map.  
Denote by
\begin{equation}
\mathcal{P}_{j}: \pi_{r+j}(D^{r}, \partial D^{r}) \cong \pi_{r+j}(S^{r}) \stackrel{\cong} \longrightarrow \Omega^{\fr}_{j}(\text{pt.})
\end{equation}
the \textit{Pontryagin-Thom} isomorphism.
The following lemma shows
how to express
$\alpha_{j}(f\circ\varphi, \; g; \; M)$
in terms of $\alpha_{0}(f, g, M)$ and the element $\mathcal{P}_{j}([\varphi]) \in \Omega^{\fr}_{j}(\text{pt.})$.
\begin{lemma} \label{lemma: hopf map trick}
Let $g$, $f$, and $\varphi$ be exactly as above and suppose that $f(\partial D^{r})\cap g(\partial Q) = \emptyset$. 
Then
$$\alpha_{j}(f\circ \varphi, \; g; \; M) \; = \; \alpha_{0}(f, g; M)\cdot \mathcal{P}_{j}([\varphi]),$$
where the product on the right-hand side is the product in the graded bordism ring $\Omega^{\fr}_{*}(\text{pt.})$. 
\end{lemma}
\begin{proof}
Let $\ell \in \Z$ denote the oriented, algebraic intersection number associated to the intersection of $f(D^{r})$ and $g(Q)$.
By application of the Whitney trick, we may deform $f$ so that
\begin{equation} \label{eq: pre-image equality}
f(D^{r})\cap g(Q) = \{x_{1}, \dots, x_{\ell}\},
\end{equation}
where the points $x_{i}$ for $i = 1, \dots, \ell$ all have the same sign. 
It follows that, 
$(f\circ\varphi)^{-1}(g(Q)) =
\bigsqcup_{i=1}^{\ell}\varphi^{-1}(x_{i}).$ 
For each $i \in \{1,
\dots, \ell\}$, the framing at $x_{i}$ (induced by the orientations of
$f(D^{r})$, $g(Q)$ and $M$) induces a framing on
$\varphi^{-1}(x_{i})$.  We denote the element of
$\Omega_{j}^{\fr}(\text{pt.})$ given by $\varphi^{-1}(x_{i})$ with
this induced framing by $[\varphi^{-1}(x_{i})]$.  By definition of the
Pontryagin-Thom map $\mathcal{P}_{j}$,
the element $[\varphi^{-1}(x_{i})]$ is equal to
$\mathcal{P}_{j}([\gamma])$ for $i = 1, \dots, \ell$.  Using the
equality (\ref{eq: pre-image equality}), it follows that
$\alpha_{j}(f\circ\varphi, \; g \; M) = \ell\cdot
\mathcal{P}_{j}([\varphi]).$ 
The proof then follows from the fact
that $\alpha_{0}(f, g, M)$ is identified with the algebraic
intersection number associated to $f(S^{r})$ and $g(Q)$.
\end{proof}
We apply the above lemma to the following proposition.
\begin{proposition} \label{proposition: intersection creation}
Let $Q$ and $M$ have non-empty boundary and let $g: (Q, \partial Q)
\longrightarrow (M, \partial M)$ be a smooth embedding.  Let $r$
denote the integer $m-q$.  There exists an embedding
$f: (D^{r}, \partial D^{r}) \longrightarrow (M, \partial M)$
that satisfies the following conditions:
\begin{itemize} \itemsep.1cm
\item $f(\partial D^{r})\cap g(\partial Q) = \emptyset$, 
\item $f(D^{r})\cap g(Q)$ consists of a single point with positive orientation, 
\item $f$ represents the trivial element in $\pi_{r}(M, \partial M)$. \end{itemize}
\end{proposition}
\begin{proof}
We will prove the proposition by carrying out an explicit construction as follows:
\begin{enumerate} \itemsep.1cm
\item[\textcolor{black}{(i)}] Choose a collar embedding $h: \partial Q\times[0, \infty) \longrightarrow Q$ such that $h^{-1}(\partial Q) = \partial Q\times\{0\}$.
\item[\textcolor{black}{(ii)}]  
Choose a point $y \in \partial Q$, then define an embedding
$\gamma: [0, 1] \longrightarrow g(Q),$ $\gamma(t) = g(h(y, t)).$ 
We then let $x \in g(Q)$ denote the point $\gamma(1)$.
\item[\textcolor{black}{(iii)}] 
Choose an embedding $\alpha: (D^{2}_{+}, \partial_{0}D^{2}_{+}) \longrightarrow (M, \partial M)$ that satisfies the following conditions:
\begin{enumerate} \itemsep.1cm
\item[(a)] $\alpha(D^{2}_{+})\cap g(Q) = \gamma([0, 1])$, 
\item[(b)] $\alpha(\partial_{1}D^{2}_{+})\cap g(Q) = \{x\}$, 
\item[(c)] $\alpha(D^{2}_{+})$ intersects $g(Q)$ orthogonally (with resect to some metric on $M$).
\end{enumerate}
\item[\textcolor{black}{(iv)}] Let $r$ denote the integer $m-q$. 
Choose a $(r-1)$-frame of orthogonal vector fields $(v_{1}, \dots, v_{r-1})$ over the embedded half-disk $\alpha(D^{2}_{+}) \subset M$ with
the property that $v_{i}$ is orthogonal to $\alpha(D^{2}_{+})$ and orthogonal to $g(Q)$ over the intersection $\alpha(D^{2}_{+})\cap g(Q)$, for $i = 1, \dots, r-1$. 
Since the disk is contractable, there is no obstruction to the existence of such a frame. 

\end{enumerate}The orthogonal $(r-1)$-frame chosen in step \textcolor{black}{(iv)} induces an embedding 
$\bar{f}: (D^{r+1}_{+}, \partial_{0}D^{r+1}_{+}) \longrightarrow (M, \partial M).$ 
The orthogonality condition (condition (c)) in Step
(iii) of the above construction, together with the orthogonality
condition on the frame chosen in step (iv), implies that
$\bar{f}(D^{r+1}_{+})$ is transverse to $g(Q)$.
Furthermore, condition (b) from step (iii) of the above construction implies that
$f(\partial_{1}D^{r+1})\cap g(\Int(Q)) = \{x\}.$  
We then set the map $f: (D^{r}, \partial D^{r-1}) \longrightarrow (M, \partial M)$ equal to the embedding obtained by restricting $\bar{f}$ to $(\partial_{1}D^{r+1}, \partial_{0,1}D^{r+1})$.
This concludes the proof of the proposition. 
\end{proof}
\subsection{Proof of Theorem \ref{theorem: boundary disjunction}} \label{subsection: proof of main theorem}
Let 
$f: (P, \partial P) \longrightarrow (M, \partial M)$ and $g: (Q, \partial Q) \longrightarrow (M, \partial M)$
be smooth embeddings exactly as in the statement of Theorem \ref{theorem: boundary disjunction}.
Suppose that conditions (i), (ii), and (iii) from the statement of Theorem \ref{theorem: boundary disjunction} are satisfied.
We now have all of the necessary tools available to prove the theorem. 
\begin{proof}[Proof of Theorem \ref{theorem: boundary disjunction}]
By Proposition \ref{proposition: boundary disjunction} we may assume that $f(\partial P)\cap g(\partial Q) = \emptyset$.
Consider the element 
$\alpha_{t}(f, g, M) \in \Omega^{\fr}_{t}(\text{pt})$
where as before $t = p + q - m$. 
Let 
$\varphi: (D^{p}, \partial D^{p}) \longrightarrow (D^{m-q}, \partial D^{m-q})$
be a map such that $\mathcal{P}_{t}([\varphi]) = -\alpha_{t}(f, g, M)$ as elements of $\Omega^{\fr}_{t}(\text{pt.})$. 
By Proposition \ref{proposition: intersection creation} there exists a null-homotopic embedding
$\phi: (D^{m-q}, \partial D^{m-q}) \longrightarrow (M, \partial M)$
such that $\phi(D^{m-q})$ intersects the interior of $g(Q)$ at exactly one point. 
Furthermore, by general position we may assume that $\phi(\partial D^{m-q})$ is disjoint from $g(\partial Q)$. 
It follows from Lemma \ref{lemma: hopf map trick} that, 
\begin{equation} \label{equation: pontryagin thom intersection}
\alpha_{t}(\phi\circ\varphi, g, M) = \mathcal{P}_{t}(\varphi)\cdot\alpha_{0}(\phi, g, M) = -\alpha_{t}(f, g, M).
\end{equation}
Now, the map $\phi\circ\varphi$ has image disjoint from $g(\partial Q) \subset \partial M$. 
Let $M'$ denote the complement $M\setminus g(\partial Q)$. 
By the connectivity and dimensional conditions from the statement of Theorem \ref{theorem: boundary disjunction}, we may apply Lemma \ref{lemma: basic embedding lemma} (see also \cite[Theorem 1]{H 69}) to obtain a homotopy 
\begin{equation} \label{equation: homotopy to embedding}
\widehat{\varphi}_{s}: (D^{p}, \partial D^{p}) \longrightarrow (M', \partial M'), \quad s \in [0,1],
\end{equation}
with $\widehat{\varphi}_{0} = \phi\circ\varphi$ such that $\widehat{\varphi}_{1}$ is an embedding. 
Let us denote,
$\widehat{\varphi} := i_{M'}\circ\widehat{\varphi}_{1}: (D^{p}, \partial D^{p}) \longrightarrow (M, \partial M)$,
where $i_{M'}: (M', \partial M') \hookrightarrow (M, \partial M)$ is the inclusion.
Since the homotopy in (\ref{equation: homotopy to embedding}) was through maps with image in $(M', \partial M')$, it follows that 
$\alpha_{t}(\widehat{\varphi}, g, M) = \alpha_{t}(\phi\circ\varphi, g, M),$
and then by the above calculation (\ref{equation: pontryagin thom intersection}) we have, 
$\alpha_{t}(\widehat{\varphi}, g, M) = -\alpha_{t}(f, g, M).$
Now, let $\widehat{f}: (P, \partial P) \longrightarrow (M, \partial M)$ be the embedding obtained by forming the boundary-connected-sum of of the submanifolds 
$$(f(P), f(\partial P)), \; (\widehat{\varphi}(D^{p}), \widehat{\varphi}(\partial D^{p})) \; \subset \; (M, \partial M),$$
along an arc in $\partial M$ that is disjoint from $g(\partial Q)$. 
 Clearly this embedding is homotopic (as a map) to $f$.
 We emphasize that homotopy taking $\widehat{f}$ to $f$ will not be constant on the boundary of $P$. 
 We then have,
 $$
\alpha_{t}(\widehat{f}, g, M) = \alpha_{t}(f, g, M) + \alpha_{t}(\widehat{\varphi}, g, M) = \alpha_{t}(f, g, M) - \alpha_{t}(f, g, M) = 0.
 $$ By Theorem \ref{theorem: hatcher quinn main theorem} there is a
diffeotopy (relative the boundary of $M$) that pushes $\widehat{f}(P)$
off of $g(Q)$.  Now since $f$ is homotopic to $\widehat{f}$, it
follows that $f$ is homotopic (through maps sending $\partial P$ to
$\partial M$) to an embedding with image disjoint from $g(\partial
Q)$.  We then apply \cite[Theorem 1.1]{HQ 74} to conclude that the
embedding $f$ is actually isotopic (rather than just homotopic) to such an embedding with image
disjoint from $g(Q)$. 
Then Theorem \ref{theorem: boundary disjunction} follows by isotopy extension.
\end{proof}

\end{document}